\documentclass[11pt]{amsart}
\usepackage{amsmath}
\usepackage{amsthm}
\usepackage{graphicx}

\newtheorem{theorem}{Theorem}
\newtheorem{proposition}[theorem]{Proposition}
\newtheorem{lemma}[theorem]{Lemma}
\newtheorem{corollary}[theorem]{Corollary}

\theoremstyle{remark}

\newtheorem*{example}{Example}

\theoremstyle{definition}

  \def\e{\epsilon}
   \def\vol{{\rm vol}}
\def\v{{\rm v}}
\def\w{{\rm w}} 
 \def\b{\mathcal{B}(C_n)}
  \def\bd{\mathcal{B}(D_n)}
 \def\c{\mathcal{B}^c(C_n)}
  \def\g{\mathcal{B}^{\beta}(C_n)}
  \def\gd{\mathcal{B}^{\beta}(D_n)}
  \def\r{\mathcal{B}^{\beta}}
\def\t{\mathcal{S}(C_n)}
\def\q{\mathbb{Q}}

\def\gr{\mathcal{G}}
\def\base{\mathcal{B}}

\begin{document}
\title[Root polytopes, triangulations, and the subdivision algebra, II]{Root polytopes, triangulations, and the subdivision algebra, II}
\author{Karola M\'esz\'aros }
\address{
Department of Mathematics, Massachusetts Institute of Technology, Cambridge, MA 02139
}
\date{June 30, 2009}

\begin{abstract}
 
  The  type $C_{n}$ root polytope $\mathcal{P}(C_{n}^+)$ is the convex hull in $\mathbb{R}^{n}$ of the origin and the points $e_i-e_j, e_i+e_j, 2e_k$ for $1\leq i<j \leq n, k \in [n]$.  Given a graph $G$, with edges labeled positive or negative, associate to each edge $e$ of $G$ a vector $\v(e)$ which is $e_i-e_j$ if $e=(i, j)$, $i<j$, is labeled negative and $e_i+e_j$ if it is labeled positive. For such a signed graph $G$,  the associated   root polytope $\mathcal{P}(G)$    is the intersection of  $\mathcal{P}(C_{n}^+)$ with the cone generated by the  vectors  $\v(e)$, for edges $e$ in $G$.   The reduced forms of a certain monomial $m[G]$ in commuting variables $x_{ij}, y_{ij}, z_k$   under   reductions derived from the relations of a bracket algebra of type $C_n$,  can be interpreted as triangulations of $\mathcal{P}(G)$.  Using these triangulations, the volume of $\mathcal{P}(G)$ can be calculated.  If we allow variables  to commute  only  when    all  their indices are distinct, then we prove that the reduced form of $m[G]$, for ``good" graphs $G$, is unique and yields  a canonical triangulation of  $\mathcal{P}(G)$  in which each simplex corresponds to a noncrossing alternating graph in a type $C$ sense.  A special case of our results proves a conjecture of  A. N. Kirillov about the uniqueness of the reduced form of a Coxeter type element in the bracket algebra of type $C_n$. We also study the bracket algebra of type $D_n$ and show that a family of monomials has unique reduced forms in it. A special case of our results proves a conjecture of  A. N. Kirillov about the uniqueness of the reduced form of a Coxeter type element in the bracket algebra of type $D_n$.
  \end{abstract}

\maketitle
  
  \section{Introduction}
\label{sec:in}

In this paper we develop the connection between triangulations of type $C_n$ root polytopes and a commutative algebra $\t$, the subdivision algebra of  type $C_n$ root polytopes. A type $C_n$ root polytope   is a convex hull  of the origin and  some of  the points $e_i-e_j, e_i+e_j, 2e_k$ for $1\leq i<j \leq n, k \in [n]$, where $e_i$ denotes the $i^{th}$ standard basis vector in $\mathbb{R}^n$.  A polytope $\mathcal{P}(m)$  corresponds to each monomial $m\in \t$, and each relation of the algebra  equating a monomial with three others, $m_0=m_1+m_2+m_3$, can be interpreted as cutting the polytope $\mathcal{P}(m_0)$ into   two polytopes $\mathcal{P}(m_1)$ and $\mathcal{P}(m_2)$ with interiors disjoint such that  $\mathcal{P}(m_1)\cap \mathcal{P}(m_2)=\mathcal{P}(m_3)$; thus the name subdivision algebra for $\t$.

 A subdivision algebra $\mathcal{S}(A_{n})$ for type $A_{n}$ root polytopes was studied in \cite{kar} yielding an exciting interplay between polytopes and algebras. The algebra $\mathcal{S}(A_{n})$ is related to the algebras studied by Fomin and Kirillov in \cite{fk} and by Kirillov in \cite{k1}, which have tight connections to Schubert calculus.  Using techniques for polytopes,  the algebra $\mathcal{S}(A_{n})$ can be understood better, and using the properties of $\mathcal{S}(A_{n})$ results for root polytopes can be deduced.  The subdivision algebra $\t$ is a type $C_n$ generalization of $\mathcal{S}(A_{n})$ and its intimate connection to type $C_n$ root polytopes is displayed by a variety of results obtained by using this connection.

  Root polytopes were first defined by Postnikov in \cite{p1}, although the full root polytope of type $A_{n}$ already appeared in the work of Gelfand, Graev and Postnikov \cite{GGP}, where they gave a canonical triangulation of it into simplices corresponding to noncrossing alternating trees. Properties of this triangulation are studied in \cite[Exercise 6.31]{ec2}. Canonical triangulations for a family of type $A_{n}$ root polytopes were constructured in \cite{kar} extending the result of \cite{GGP}. In this paper we define type $C_n$ analogs for noncrossing and alternating graphs, and show that a family of type $C_n$ root polytopes, containing the full root polytope, has canonical triangulations into simplices corresponding to noncrossing alternating graphs. Using the canonical  triangulations we compute the volumes for these root polytopes.
 
 The subdivision algebra $\t$ is closely related to the noncommutative bracket algebra $\b$ of type $C_n$ defined by A. N. Kirillov \cite{kir}.   Kirillov conjectured  the uniqueness of the reduced form of a Coxeter type element in $\b$.  As the algebras $\t$ and  $\b$ have over ten not-so-simple-looking relations, we postpone their definitions and the precise statement of Kirillov's conjecture till Section \ref{sec:alg}. While at the first sight the relations of $\b$ might appear rather mysterious,  we interpret them similarly to the relations of $\t$, as certain subdivisions of root polytopes. 
 This connection ultimately  yields a proof of Kirillov's  conjecture along with more general theorems on reduced forms, of which there are two types.  In the noncommutative algebra $\b$ we show that for  a family of monomials $\mathcal{M}$, including the Coxeter type element defined by Kirillov, the reduced form is unique. In the commutative algebra $\t$ and the commutative counterpart $\c= \b / [\b, \b] $ of $\b$, the reduced forms are not unique; however, we show that the number of monomials in a reduced form of $m \in \mathcal{M}$  is independent of the order of reductions  performed. 
 
 We also study  the noncommutative bracket algebra $\bd$ of type $D_n$ defined by A. N. Kirillov \cite{kir}. Using noncommutative Gr\"obner bases techniques we prove that a family of monomials has unique reduced forms in it. A special case of our results proves a conjecture of  A. N. Kirillov about the uniqueness of the reduced form of a Coxeter type element in the bracket algebra of type $D_n$. 
  
 This paper is organized as follows. In Section \ref{sec:alg} we give the definition of $\b$, as well as two related commutative algebras $\c$ and $\t$. We also state Kirillov's conjecture pertaining to $\b$ in Section \ref{sec:alg}.  In Section \ref{sec:red} we introduce signed graphs, define the type $C$ analogue of alternating graphs, and  show how to reformulate the relations of   the algebras $ \c, \t$ into reductions on graphs. In Section \ref{sec:root} we  introduce coned root polytopes of type $C_n$ and state  the Reduction Lemma which connects root polytopes and  the algebras $\b, \c, \t$. In Section \ref{sec:play} we prove a characterization of the vertices of coned type $C_n$ root polytopes, while in Section \ref{sec:reduction_lemma} we prove the Reduction Lemma. In Section \ref{sec:forest} we establish the relation between volumes of root polytopes and reduced forms of monomials in the algebras $ \c, \t$ using the Reduction Lemma. In Section \ref{reductionsB} we reformulate the noncommutative relations of $\b$ in terms of egde-labeled graphs and define well-structured and well-labeled graphs, key for our further considerations. In Section \ref{reductionsB1} we prove a simplified version of Kirillov's conjecture, construct a canonical triangulation for the full type $C_n$ root polytope $\mathcal{P}(C_n^+)$ and calculate its volume. In Section \ref{sec:gen} we generalize Kirillov's conjecture  to all monomials arising from well-structured and well-labeled graphs and give the triangulations and volumes of the corresponding root polytopes. In Section \ref{11}  we prove the general form of Kirillov's conjecture in a weighted bracket algebra $\g$,  and show a way to calculate Ehrhart polynomials of certain type $C_n$ root polytopes.   In Section \ref{sec:D} the definition of $\bd$ is given along with Kirillov's conjecture 
  pertaining to it. In Section \ref{sec:13}  combinatorial results regarding a family of monomials are proved. Finally, in Section \ref{sec:grobi} we prove a general result on the reduced forms of monomials implying Kirillov's type $D_n$ conjecture. 
 
 \section{The bracket and subdivision algebras of type $C_n$}
 \label{sec:alg}
 
 In this section the definition of the bracket algebra $\b$ is given, along with a conjecture of Kirillov pertaining to it. We introduce the subdivision algebra $\t$, which, as its name suggests, will be shown to govern subdivisions of type $C_n$ root polytopes.  
 
 Kirillov \cite{kir} defined the algebra we are denoting $\b$ as a type $B_n$ bracket algebra $\mathcal{B}(B_n)$, but since we can interpret its generating  variables as corresponding to either the type $B_n$ and type $C_n$ roots, we refer to  it as a type $C_n$ bracket algebra $\b$. The reason for our desire to designate $\b$ as a type $C_n$ algebra is its essential  link  to type $C_n$ root polytopes, which we develop in this paper. Here we define   a simplified form of the bracket algebra $\b$; for a more general definition, see Section \ref{11}.
 
 Let the {\bf bracket algebra} $\mathcal{B}(C_n)$ {\bf  of type $C_n$} be an associative algebra  over $\mathbb{Q}$  with a set of generators  $\{x_{ij}, y_{ij}, z_i \mid 1 \leq i\neq j\leq n\}$ subject to the following relations:
 
 (1) $x_{ij}+x_{ji}=0,$ $y_{ij}=y_{ji}$, for $i \neq j$,

(2) $z_i z_j=z_j z_i$

($3$) $x_{ij}x_{kl}= x_{kl}x_{ij}$, $y_{ij}x_{kl}= x_{kl}y_{ij}$, $y_{ij}y_{kl}= y_{kl}y_{ij}$, for  $i <j, k<l$ distinct.

($4$) $z_i x_{kl}=x_{kl} z_i$, $z_i y_{kl}=y_{kl} z_i$, for all $i\neq k, l$

(5) $x_{ij}x_{jk}=x_{ik}x_{ij}+x_{jk}x_{ik}$,  for  $1\leq i<j<k\leq n$,

($5'$) $x_{jk}x_{ij}=x_{ij}x_{ik}+x_{ik}x_{jk}$, for  $1\leq i<j<k\leq n$,

(6) $x_{ij}y_{jk}=y_{ik}x_{ij}+y_{jk}y_{ik}$,  for  $1\leq i<j<k\leq n$,

($6'$) $y_{jk}x_{ij}=x_{ij}y_{ik}+y_{ik}y_{jk}$, for  $1\leq i<j<k\leq n$,

(7) $x_{ik}y_{jk}=y_{jk}y_{ij}+y_{ij}x_{ik}$, for  $1\leq i<j<k\leq n$,

($7'$) $y_{jk}x_{ik}=y_{ij}y_{jk}+x_{ik}y_{ij}$, for  $1\leq i<j<k\leq n$,

(8) $y_{ik}x_{jk}=x_{jk}y_{ij}+y_{ij}y_{ik}$, for  $1\leq i<j<k\leq n$,

($8'$) $x_{jk}y_{ik}=y_{ij}x_{jk}+y_{ik}y_{ij}$, for  $1\leq i<j<k\leq n$,

(9) $x_{ij}z_j=z_i x_{ij}+ y_{ij} z_i + z_j y_{ij}$, for $i<j$

($9'$) $z_jx_{ij}= x_{ij}z_i+  z_i y_{ij}+  y_{ij}z_j$, for $i<j$



 \medskip
 
Let $w_{C_n}=\prod_{i=1}^{n-1} x_{i, i+1}z_n$ be a Coxeter type element in $\mathcal{B}(C_n)$ and let $P^\mathcal{B}_n$ be the polynomial in variables $x_{ij}, y_{ij}, z_i,  1 \leq i\neq j\leq n$  obtained from $w_{C_n}$ by successively applying the defining relations $(1)-(9')$ in any order until unable to do so. 
We call $P^\mathcal{B}_n$  a {\bf reduced form} of $w_{C_n}$ and consider the process of  successively applying the defining relations $(5)-(9')$ as a reduction process, with possible commutations (2)-(4) between reductions, as we show in the following example.

\begin{eqnarray}  \label{ex2}
  \mbox {\boldmath$  x_{12}x_{23}$}z_3 & \rightarrow&  x_{13}\underline{x_{12} z_3}+x_{23}\mbox{\boldmath $x_{13} z_3$}  \nonumber \\
&\rightarrow& \mbox{\boldmath $x_{13} z_3$} x_{12} +x_{23}z_1 x_{13}  +\mbox{\boldmath $x_{23}y_{13}$} z_1+  \mbox{\boldmath $x_{23}z_3$} y_{13}    \nonumber \\
&\rightarrow &  z_1 x_{13} x_{12} +  y_{13} z_1 x_{12} +   z_3 y_{13}  x_{12}+x_{23}z_1 x_{13}   +y_{12}x_{23}  z_1  +y_{13}y_{12} z_1     \nonumber \\ & &  + z_2 \mbox{\boldmath $ x_{23} y_{13} $}+ y_{23} z_2y_{13} +z_3 y_{23}  y_{13}      \nonumber \\ 
&\rightarrow &  z_1 x_{13} x_{12} +  y_{13} z_1 x_{12} +   z_3 y_{13}  x_{12}+x_{23}z_1 x_{13}  +y_{12}x_{23}  z_1 \nonumber \\ & & +y_{13}y_{12} z_1      + z_2 y_{12}x_{23}+z_2 y_{13}y_{12}  + y_{23} z_2y_{13} +z_3 y_{23}  y_{13}    \nonumber
\end{eqnarray}

In the example above the pair of variables on which one of reductions $(5)-(9')$ is performed is in boldface, and the variables which we commute according to one of   (2)-(4)  are underlined.

 \medskip
\noindent {\bf Conjecture 1. (Kirillov \cite{kir})}\textit{ Apart from applying the  relations (1)-(4), the reduced form  $P^\mathcal{B}_n$ of $w_{C_n}$ does not depend on the order in which the reductions are performed.}
 \medskip
 
 Note that the above statement does not hold true for any monomial. We show one simple example of how it fails.

\begin{eqnarray}  \label{ex1}
  \mbox {\boldmath$  x_{12}x_{23}$}y_{13} & \rightarrow&  x_{13} x_{12} y_{13}+x_{23} x_{13} y_{13}   
  \end{eqnarray}

\begin{eqnarray}  \label{ex2}
  x_{12}  \mbox {\boldmath$x_{23}y_{13}$} & \rightarrow&  x_{12}  y_{12} x_{23}+x_{12}  y_{13} y_{12}   
  \end{eqnarray}
  
 Note that we reduced the monomial $x_{12}x_{23}y_{23}$ in two different ways yielding two different polynomials. The reader can also check another example of this phenomenon by reducing the monomial $y_{14}x_{24}y_{34}$ in two different ways to obtain two different reduced forms.

We prove Conjecture 1 in Section \ref{reductionsB1}, as well as its generalizations in Sections \ref{sec:gen} and \ref{11}. We first define and study a commutative algebra $\t$ closely related to $\b$, though more complicated than its commutative counterpart, $\c=\b/ [\b, \b]$, which is simply  the commutative associative algebra over $\mathbb{Q}$  with a set of generators  $\{x_{ij}, y_{ij}, z_i \mid 1 \leq i\neq j\leq n\}$ subject to  relations  (1) and $(5)-(9')$ from above. Our motivation for defining $\t$ is a natural correspondence between the relations of $\t$ and ways to subdivide type $C_n$ root polytopes, which   correspondence is made precise in the Reduction Lemma (Lemma \ref{reduction_lemma}). In order to emphasize this  connection, we call $\t$ the  subdivision algebra of type $C_n$. The subalgebra $\mathcal{S}(A_{n-1})$ of $\t$ generated by  $\{x_{ij} \mid 1 \leq i\neq j\leq n\}$ has been studied in \cite{kar}, and an analogous correspondence between the relations of  $\mathcal{S}(A_{n-1})$ and ways to subdivide type $A_{n-1}$ root polytopes has been established.  Moreover,  results in the spirit of Conjecture 1 for type $A_{n-1}$ can also be found in \cite{kar}.

 Let the {\bf subdivision algebra}  $\t$ be the commutative  algebra  over $\mathbb{Q}[\beta]$, where $\beta$ is a variable (and a central element),    with a set of generators  $\{x_{ij}, y_{ij}, z_i \mid 1 \leq i\neq j\leq n\}$ subject to the following relations:
 
 (1) $x_{ij}+x_{ji}=0,$ $y_{ij}=y_{ji}$, for $i \neq j$,

(2) $x_{ij}x_{jk}=x_{ik}x_{ij}+x_{jk}x_{ik}+\beta x_{ik}$,  for  $1\leq i<j<k\leq n$,

(3) $x_{ij}y_{jk}=y_{ik}x_{ij}+y_{jk}y_{ik}+\beta y_{ik}$,  for  $1\leq i<j<k\leq n$,

(4) $x_{ik}y_{jk}=y_{jk}y_{ij}+y_{ij}x_{ik}+\beta y_{ij}$, for  $1\leq i<j<k\leq n$,

(5) $y_{ik}x_{jk}=x_{jk}y_{ij}+y_{ij}y_{ik}+\beta y_{ij}$, for  $1\leq i<j<k\leq n$,

(6) $y_{ij}x_{ij}=z_i x_{ij}+ y_{ij} z_i +\beta z_i$, for $i<j$

(7) $x_{ij}z_j=y_{ij}x_{ij} + z_j y_{ij} + \beta y_{ij} $, for $i<j.$

 Notice that when we set $\beta=0$ relations (2)-(5) of $\t$ become relations (5)-(8) of $\b$, and if we combine relations (6) and (7) of $\t$ we obtain relation $(9)$ of $\b$.  In some cases we will in fact simply work with the  commutative counterpart of $\b$, $\c$.

We treat  relations (2)-(7) of $\t$ as {\bf reduction rules}:  
\begin{equation} \label{red1}
x_{ij}x_{jk}\rightarrow x_{ik}x_{ij}+x_{jk}x_{ik}+\beta x_{ik},
 \end{equation}
\begin{equation} \label{red2}
 x_{ij}y_{jk}\rightarrow y_{ik}x_{ij}+y_{jk}y_{ik}+\beta y_{ik},
  \end{equation}
\begin{equation} \label{red3}
x_{ik}y_{jk}\rightarrow y_{jk}y_{ij}+y_{ij}x_{ik}+\beta y_{ij},
 \end{equation}
\begin{equation} \label{red4}
y_{ik}x_{jk}\rightarrow x_{jk}y_{ij}+y_{ij}y_{ik}+\beta y_{ij}. 
\end{equation}
\begin{equation} \label{red5}
y_{ij}x_{ij}\rightarrow  z_i x_{ij}+ y_{ij} z_i +\beta z_i 
\end{equation}
\begin{equation} \label{red6}
x_{ij}z_j\rightarrow y_{ij}x_{ij} + z_j y_{ij} + \beta y_{ij}  
\end{equation}
 \vskip0.1in
 
 A  \textbf{reduced form} of the monomial $m$ in variables  $x_{ij}, y_{ij}, z_k,  1 \leq i< j\leq n, k \in [n],$ in the algebra $\t$  is a polynomial $P_n^\mathcal{S}$  obtained by successive applications of reductions (\ref{red1})-(\ref{red6}) until no further reduction is possible, where we allow commuting any two variables. Requiring that   $m$ is in variables  $x_{ij}, y_{ij}, z_k,  1 \leq i< j\leq n, k \in [n],$ is without loss of generality, since otherwise we can simply  replace $x_{ij}$ with $-x_{ji}$ and $y_{ij}$ with $y_{ji}$. Note that the reduced forms are not necessarily unique. However  we show in Section \ref{sec:forest} that the number of monomials in a reduced form of a suitable monomial $m$ is independent of the order of the reductions performed.

   \section{Commutative reductions in terms of graphs}
\label{sec:red}


 
In this section we rephrase the reduction process described in Section \ref{sec:alg}  in terms of graphs. This view will be useful throughout the paper.  We use the language of signed graphs. Signed graphs have appeared in the literature before, for example in Zaslavsky's and Reiner's work \cite{z1, z2, R1, R2}. Their notation is not the same, and we use a notation closer to Reiner's. In particular, positive and negative edges in our notation  mean something different than in Zaslavsky's language. We request the reader to read the definitions with full attention for this reason.
 
A \textbf{signed graph} $G$ on the vertex set $[n]$ is a multigraph with each edge labeled by $+$ or $-$.   All graphs in this paper are signed and in each of them the loops are labeled positive. We denote an edge with endpoints $i, j$ and sign $\e \in \{+, -\}$ by $(i, j, \e)$. Note that  $(i, j, \e)=(j, i, \e)$. As a result, we drop the signs from the loops in figures.  
A positive edge, that is an edge labeled by $+$, is said to be {\bf positively incident}, or, {\bf incident with a positive sign}, to both of its endpoints.  A negative  edge is positively incident to its smaller vertex and  {\bf negatively  incident} to  its greater endpoint.
We say that a graph is {\bf alternating} if for any vertex $v\in V(G)$ the edges of $G$ incident to $v$ are incident to $v$ with the same sign.

Think of a monomial $m \in \t$  in variables $x_{ij}, y_{ij}, z_k, 1\leq i<j \leq n, k \in [n],$ as a signed graph $G$ on the vertex set $[n]$ with a negative edge $(i, j, -)$ for each appearance of $x_{ij}$ in $m$ and with a positive edge $(i, j, +)$ for each appearance of $y_{ij}$ in $m$ and with a loop $(i, i, +)$ for each appearance of $z_{i}$ in $m$. Let $G^\mathcal{S}[m]$ denote this graph. 
   It is straighforward to reformulate the  reduction rules (\ref{red1})-(\ref{red6}) in terms of reductions on graphs. If $m \in \t$, then we replace each monomial $m$ in the reductions by corresponding graphs $G^\mathcal{S}[m]$. 
   
   {\bf Reduction rules for graphs:} 
   
   Given   a graph $G_0$ on the vertex set $[n]$ and   $(i, j, -), (j, k, -) \in E(G_0)$ for some $i<j<k$, let   $G_1, G_2, G_3$ be graphs on the vertex set $[n]$ with edge sets
  \begin{eqnarray} \label{graphs1}
E(G_1)&=&E(G_0)\backslash \{(j, k,-)\} \cup \{(i, k, -)\}, \nonumber \\
E(G_2)&=&E(G_0)\backslash \{(i, j,-)\} \cup \{(i, k,-)\},\nonumber \\ 
E(G_3)&=&E(G_0)\backslash \{(i, j,-)\} \backslash \{(j, k,-)\} \cup \{(i, k,-)\}. 
\end{eqnarray}
    Given   a graph $G_0$ on the vertex set $[n]$ and   $(i, j, -), (j, k, +) \in E(G_0)$ for some $i<j<k$, let   $G_1, G_2, G_3$ be graphs on the vertex set $[n]$ with edge sets
  \begin{eqnarray} \label{graphs2}
E(G_1)&=&E(G_0)\backslash \{(j, k,+)\} \cup \{(i, k, +)\}, \nonumber \\
E(G_2)&=&E(G_0)\backslash \{(i, j,-)\} \cup \{(i, k, +)\},\nonumber \\ 
E(G_3)&=&E(G_0)\backslash \{(i, j,-)\} \backslash \{(j, k, +)\} \cup \{(i, k, +)\}. 
\end{eqnarray}
   Given   a graph $G_0$ on the vertex set $[n]$ and   $(i, k, -), (j, k, +) \in E(G_0)$ for some $i<j<k$, let   $G_1, G_2, G_3$ be graphs on the vertex set $[n]$ with edge sets
  \begin{eqnarray} \label{graphs3}
E(G_1)&=&E(G_0)\backslash \{(j, k, +)\} \cup \{(i, j, +)\}, \nonumber \\
E(G_2)&=&E(G_0)\backslash \{(i, k, -)\} \cup \{(i, j, +)\},\nonumber \\ 
E(G_3)&=&E(G_0)\backslash \{(i, k, -)\} \backslash \{(j, k, +)\} \cup \{(i, j, +)\}. 
\end{eqnarray}
   Given   a graph $G_0$ on the vertex set $[n]$ and   $(i, k, +), (j, k, -) \in E(G_0)$ for some $i<j<k$, let   $G_1, G_2, G_3$ be graphs on the vertex set $[n]$ with edge sets
  \begin{eqnarray} \label{graphs4}
E(G_1)&=&E(G_0)\backslash \{(j, k, -)\} \cup \{(i, j, +)\}, \nonumber \\
E(G_2)&=&E(G_0)\backslash \{(i, k, +)\} \cup \{(i, j, +)\},\nonumber \\ 
E(G_3)&=&E(G_0)\backslash \{(i, k, +)\} \backslash \{(j, k, -)\} \cup \{(i, j, +)\}. 
\end{eqnarray}

 Given   a graph $G_0$ on the vertex set $[n]$ and   $(i, j, -), (i, j, +) \in E(G_0)$ for some $i<j $, let   $G_1, G_2, G_3$ be graphs on the vertex set $[n]$ with edge sets
  \begin{eqnarray} \label{graphs5}
E(G_1)&=&E(G_0)\backslash \{(i, j, +)\} \cup \{(i, i, +)\}, \nonumber \\
E(G_2)&=&E(G_0)\backslash \{(i, j, -)\} \cup \{(i, i, +)\},\nonumber \\ 
E(G_3)&=&E(G_0)\backslash \{(i, j, +)\} \backslash \{(i, j, +)\}   \cup \{(i, i, +)\}. 
\end{eqnarray}

Given   a graph $G_0$ on the vertex set $[n]$ and   $(i, j, -), (j, j, +) \in E(G_0)$ for some $i<j $, let   $G_1, G_2, G_3$ be graphs on the vertex set $[n+1]$ with edge sets
  \begin{eqnarray} \label{graphs6}
E(G_1)&=&E(G_0)\backslash \{(j, j, +)\} \cup \{(i, j, +)\}, \nonumber \\
E(G_2)&=&E(G_0)\backslash \{(i, j, -)\} \cup \{(i, j, +)\},\nonumber \\ 
E(G_3)&=&E(G_0)\backslash \{(j, j, +)\} \backslash \{(i, j, -)\} \cup \{(i, j, +)\}. 
\end{eqnarray}

    We say that $G_0$ \textbf{reduces} to $G_1, G_2, G_3$ under the reduction rules defined by equations (\ref{graphs1})-(\ref{graphs6}).

An \textbf{$\mathcal{S}$-reduction tree} $\mathcal{T}^\mathcal{S}$  for a monomial $m_0$, or equivalently, the graph $G^\mathcal{S}[m_0]$,  is constructured as follows. 
  The root of $\mathcal{T}^\mathcal{S}$ is labeled by $G^\mathcal{S}[ m_0]$. Each node $G^\mathcal{S}[m]$ in $\mathcal{T}^\mathcal{S}$  has three children, which depend on the choice of the edges  of  $G^\mathcal{S}[m]$ on which we perform the reduction. E.g., if the reduction  is performed on edges $(i, j, -), (j, k, -) \in E(G^\mathcal{S}[m])$,  $i<j<k$, then the three children of  the node $G_0=G^\mathcal{S}[m]$   are labeled by the graphs   $G_1, G_2, G_3$   as described by equation (\ref{graphs1}). For an example of an $\mathcal{S}$-reduction tree, see Figure \ref{fig:tau}.

    \begin{figure}[htbp] 
\begin{center} 
\includegraphics[width=1.1\textwidth]{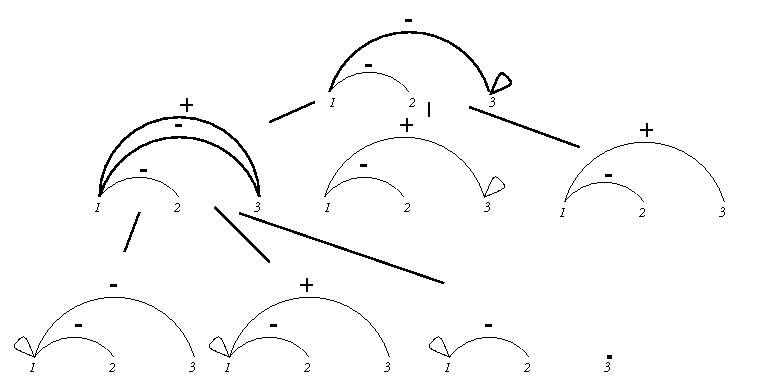} 
\caption{An $\mathcal{S}$-reduction tree with root corresponding to the monomial $x_{12}x_{13}z_3$. Summing the monomials corresponding to  the graphs labeling the  leaves  of the reduction tree   multiplied by suitable powers of $\beta$, we obtain a reduced form $P^\mathcal{S}_n$ of  $x_{12}x_{13}z_3$,  $P^\mathcal{S}_n=z_1x_{12}x_{13}+z_1x_{12}y_{13}+\beta z_1 x_{12}+x_{12}y_{13}z_3+\beta x_{12}y_{13}$.  } 
\label{fig:tau}
\end{center} 
\end{figure}

  Of course, given a graph we can also easily recover the corresponding monomial. Namely, given a graph $G$ on the vertex set $[n]$ we associate to it the monomial $m^\mathcal{S}[G]=m^{\mathcal{B}^c}[G]=\prod_{(i, j, \e) \in E(G)}\w(i, j,\e)$, where $\w(i, j, -)=x_{ij}$ for $i<j$, $\w(i, j, -)=x_{ji}$ for $i>j$, $\w(i, j,+)=y_{ij}$ and $\w(i, i,+)=z_{i}$. 
Summing the monomials corresponding to  the graphs labeling the  leaves  of the reduction tree $\mathcal{T}^\mathcal{S}$ multiplied by suitable powers of $\beta$, we obtain a reduced form of  $m_0$.

\section{Coned type $C$ root polytopes}
 \label{sec:root}
 
Generalizing the terminology of \cite[Definition 12.1]{p1}, a root polytope of   type $C_{n}$ is  the convex hull of the origin  and some of the points $e_i-e_j$,  $e_i+e_j$ and $2e_k$ for $1\leq i<j \leq n$, $k \in [n]$,  where $e_i$ denotes the $i^{th}$ coordinate vector in $\mathbb{R}^{n}$. A very special root polytope is the full type $C_n$ root polytope  \begin{align*}\mathcal{P}(C_{n}^+)&=\textrm{ConvHull}(0,  e_{ij}^-,   e_{ij}^+ , 2e_k \mid  1\leq i<j \leq n, k \in [n]) \\ &=\textrm{ConvHull}(0,  e_{ij}^-,  2e_k \mid  1\leq i<j \leq n, k \in [n]),\end{align*}  where  $e_{ij}^-=e_i-e_j$ and $e_{ij}^+=e_i+e_j$.  We study a  class of root polytopes including $\mathcal{P}(C_{n}^+)$, which we now discuss. 

Let $G$ be a   graph on the vertex set $[n]$.  Let  \[\v(i, j, \e) = \left\{ 
\begin{array}{l l}
  e_{ij}^\e & \quad \mbox{if $i\leq j$}\\
  e_{ji}^\e& \quad \mbox{if $i>j,$}\\
\end{array} \right. \]
  Define $$\mathcal{V}_G=\{\v(i, j, \e)  \mid  (i, j, \e) \in E(G)\}, \mbox{ a set of vectors associated to $G$;}$$

 $$\mathcal{C}(G)=\langle \mathcal{V}_G \rangle :=\{\sum_{\v(i, j, \e) \in \mathcal{V}_G}c_{ij} \v(i, j, \e) \mid  c_{ij}\geq 0\}, \mbox{ the {\bf cone} associated to $G$; and } $$  
  $$\overline{\mathcal{V}}_G=\Phi^+ \cap \mathcal{C}(G), \mbox{ all the positive roots of type $C_n$ contained in $\mathcal{C}(G)$}, $$
   where $\Phi^+=\{e_{ij}^- , e_{ij}^+, 2e_k \mid1\leq i<j \leq n, k \in [n]\}$ is the set of         positive roots of type $C_n$. The idea to consider the positive roots of a root system inside a cone appeared earlier in Reiner's work \cite{R1}, \cite{R2} on signed posets. Coned type $A_n$ root polytopes were studied in \cite{kar}.
   
   Define the {\bf transitive closure} of a graph $G$ as $$\overline{G}=\{(i, j, \e) \mid \v(i, j, \e) \in \overline{\mathcal{V}}_G\}$$
 
The {\bf root polytope} $\mathcal{P}(G)$ associated to graph $G$ is  
 
 \begin{equation} \label{eq1} \mathcal{P}(G)=\textrm{ConvHull}(0,  \v(i, j, \e)  \mid  (i, j, \e)  \in \overline{G})\end{equation} The root polytope $\mathcal{P}(G)$ associated to graph $G$ can also be defined as  \begin{equation} \label{eq2} \mathcal{P}(G)=\mathcal{P}(C_n^+) \cap \mathcal{C}(G).\end{equation} The equivalence of these two definition is proved in Lemma \ref{equivalent} in Section \ref{sec:reduction_lemma}. 
  
 Note that $\mathcal{P}(C_{n}^+)=\mathcal{P}(P^l)$ for the graph  $P^l=([n], \{(n,n, +), (i, i+1, -) \mid i \in [n-1]\}).$  While the choice of $G$ such that  $\mathcal{P}(C_{n}^+)=\mathcal{P}(G)$ is not unique, it becomes unique if we require that $G$ is {\bf minimal}, that is for no edge $(i, j, \e) \in E(G)$ can the corresponding vector $\v(i, j, \e)$ be written as a nonnegative linear combination of the vectors corresponding to the   edges $E(G) \backslash \{(i, j, \e)\}$.  Graph $P^l$ is minimal. 
 
 We can describe the vertices in $ \overline{\mathcal{V}}_G$ in terms of paths in $G$. 
     A \textbf{playable route} $P$ of a graph $G$ is an ordered sequence of edges $(i_1, j_1, \e_1),  \ldots,$ $(i_l, j_l, \e_l) \in E(G)$, $j_k=i_{k+1}$ for $k \in [l-1]$,  such that  $(i_k, j_k, \e_k)$ and $(i_{k+1}, j_{k+1}, \e_{k+1})$, $k \in [l-1]$, are incident to $j_k=i_{k+1}$ with opposite signs. For a playable route of $G$,   $\v(i_1, j_1,{\e_1})+ \cdots+\v(i_l, j_l,{\e_l}) \in \Phi^+$.

A \textbf{playable pair} $(P_1, P_2)$ in a graph $G$ is a pair of  playable routes $(i_1, j_1, \e_1), $ $\ldots,$ $(i_l, j_l, \e_l)$ and  $(i'_1, j'_1, \e'_1), \ldots,$ $(i'_{l'}, j'_{l'}, \e'_{l'})$ such that $i_1=j_l$ and $i'_1=j'_{l'}$.  It follows that $\frac{1}{2}(\v(i_1, j_1,{\e_1})+ \cdots+\v(i_l, j_l,{\e_l}))+\frac{1}{2}(\v(i'_1, j'_1,{\e'_1})+ \cdots+\v(i'_{l'}, j'_{l'},{\e_{l'}}))\in \Phi^+$.

Define a map $\phi$ from the playable routes and playable pairs to $ \Phi^+$ as follows.

\begin{align} \label{phieq}\nonumber
 \phi(P)&=  \v(i_1, j_1,{\e_1})+ \cdots+\v(i_l, j_l,{\e_l}),   \mbox{ where $P$ is the playable route} \\ & \nonumber\mbox{above,} \\    
 \nonumber \phi(P_1, P_2) &=\frac{1}{2}(\v(i_1, j_1,{\e_1})+ \cdots+\v(i_l, j_l,{\e_l}))+\frac{1}{2}(\v(i'_1, j'_1,{\e'_1})+ \cdots+\\&  +\v(i'_{l'}, j'_{l'},{\e_{l'}})), \mbox{ where $(P_1, P_2)$ is the playable pair above.}  
  \end{align}

 \begin{proposition} \label{playable_route}
 Let $G$ be a graph on the vertex set $[n]$. Any $v \in  \overline{\mathcal{V}}_G$ is $v=\phi(P)$ or $v=\phi(P_1, P_2)$  for some playable route  $P$ or playable pair $(P_1, P_2)$ of $G$. If  the set of vectors $\mathcal{V}_G$ is  linearly independent, then the correspondence between playable routes and pairs of $G$ and vertices in $\overline{\mathcal{V}}_G$ is a bijection.
 \end{proposition}
 
 The proof of Proposition  \ref{playable_route} appears in Section \ref{sec:play}.

 Define 
    $$\mathcal{L}_n=\{ G=([n], E(G))    \mid  \mbox{ $\mathcal{V}_G$ is  a linearly independent set}\},$$ and 
    $$\mathcal{L}(C_{n}^+)=\{ \mathcal{P}(G) \mid  G \in \mathcal{L}_n  \}, \mbox{the set of type $C_n$ \textbf{coned  root polytopes}}$$  with linearly independent generators. Since all polytopes in this paper are         coned  root polytopes   with linearly independent generators, we simply refer to them as coned root polytopes.

The next lemma characterizes graphs $G$ which belong to $ \mathcal{L}_n$; a version of it appears   in  \cite[p. 42]{fong}.

\begin{lemma}\label{fong} ( \cite[p. 42]{fong})
A   graph  $G$ on the vertex set $[n]$ belongs to $ \mathcal{L}_n$ if and only if each connected component of $G$ is a tree or a graph whose  unique simple cycle  has an odd number of positively labeled edges. 
\end{lemma}

    The full root polytope  $ \mathcal{P}(C_{n}^+) \in \mathcal{L}(C_{n}^+)$, since   the graph $P^l\in \mathcal{L}_n$ by Lemma \ref{fong}. We show below  how to obtain central triangulations for all polytopes $\mathcal{P}   \in  \mathcal{L}(C_{n}^+)$. A \textbf{central triangulation} of a $d$-dimensional  root polytope $\mathcal{P}$ is a collection  of $d$-dimensional simplices with disjoint interiors whose union is  $\mathcal{P}$, the vertices of which are vertices of $\mathcal{P}$ and the origin is a vertex of all of them. Depending on the context we at times   take the intersections of these maximal simplices to be part of the triangulation. 
    
        We now  state the crucial lemma which relates root polytopes and the algebras $\b, \c$ and $\t$ defined in Section \ref{sec:alg}.

    \begin{lemma} \label{reduction_lemma} \textbf{(Reduction Lemma)} 
Given   a graph $G_0 \in \mathcal{L}_n$ with $d$ edges let   $G_1, G_2, G_3$ be as described by any one of the equations (\ref{graphs1})-(\ref{graphs6}).   Then  $G_1, G_2, G_3 \in \mathcal{L}_n$,
$$\mathcal{P}(G_0)=\mathcal{P}(G_1) \cup \mathcal{P}(G_2)$$   where all polytopes  $\mathcal{P}(G_0), \mathcal{P}(G_1), \mathcal{P}(G_2)$ are   $d$-dimensional and    
$$\mathcal{P}(G_3)=\mathcal{P}(G_1) \cap \mathcal{P}(G_2)  \mbox{   is $(d-1)$-dimensional. } $$
\end{lemma}
      \medskip
      
       What the Reduction Lemma really says is that performing a reduction on graph $G_0 \in \mathcal{L}_n$  is the same as ``cutting" the $d$-dimensional  polytope  $\mathcal{P}(G_0)$ into two $d$-dimensional polytopes $\mathcal{P}(G_1)$ and $ \mathcal{P}(G_2)$, whose vertex  sets are  subsets of the vertex set of   $\mathcal{P}(G_0)$, whose interiors are disjoint, whose union is $\mathcal{P}(G_0)$, and whose intersection is a facet of both. We prove the Reduction Lemma in Section \ref{sec:reduction_lemma}.

 \section{Characterizing the vertices of coned root polytopes}
 \label{sec:play}
 
 In this section we prove Proposition \ref{playable_route}, which characterizes the vertices of any root polytope $\mathcal{P}(G)$. We start by proving the statement for connected $G \in \mathcal{L}_n$.

\medskip

\begin{proposition} \label{conn}
 Let $G\in \mathcal{L}_n$ be a connected graph. The correspondence between playable routes  of $G$ and vertices in $\overline{\mathcal{V}}_G$ given by $$\phi: P= \{(i_1, j_1, \e_1), (i_2, j_2, \e_2), \ldots, (i_l, j_l, \e_l)\}  \mapsto   \v(i_1, j_1,{\e_1})+ \cdots+\v(i_l, j_l,{\e_l}), $$ is a bijection.
 \end{proposition}

Denote by   $[e_i]w$  the coefficient of $e_i$ when  $w \in \mathbb{R}^n$ is expressed in terms of the standard basis  $e_1, \ldots, e_n$ of $\mathbb{R}^n$. 
\vskip 10pt 

\noindent \textit{Proof of Proposition  \ref{conn}.}
Given a playable route  $P$ of $G$, $\phi(P)\in {\overline{\mathcal{V}}}_G$ by definition.  
It remains to show that 
 for each vertex $v \in {\overline{\mathcal{V}}}_G$ there exists a playable route $P$ in $G$ such that $v=\phi(P)$. The uniqueness of such a route follows from the linear independence of the  set of vectors  $\mathcal{V}_G$ for $G\in \mathcal{L}_n$.

Consider  $v \in {\overline{\mathcal{V}}}_G$. Then $v=e_i\pm e_j$, for some $1\leq i<j\leq n$, or $v=2e_k=e_k+e_k$, for $k \in [n]$, and \begin{equation} \label{H}  v=\sum_{e\in E(G)}c_e \v(e), \mbox{for some  real $c_e\geq 0$.} \end{equation} Let $H=([n], \{ e \in E(G) \mid c_e \neq 0\})$. Observe that $H$ has at most one  connected component containing edges. This follows since a connected $G\in \mathcal{L}_n$ contains at most one simple cycle, and if there were two connected components of $H$, one would  be a tree contributing at least two nonzero coordinates to  the right hand side of (\ref{H}) and each  connected component  containing edges contributes at least one nonzero coordinate to  the right hand side of (\ref{H}).  But,   the left hand side of (\ref{H}) has one or  two nonzero coordinates.

If $k$ is a leaf of $H$ then $[e_k]v\neq 0$. Therefore, $H$ can have at most two leaves. We consider three cases depending on the number of leaves $H$ has: $0, 1, 2$. In all cases we show that there exists a playable route $P$ of $G$ with all its edges among the edges of $H$, such that $\phi(P)=v$, yielding the desired conclusion. 

\textbf{Case 1.}  $H$ has $0$ leaves. Since $H \subset G \in \mathcal{L}_n$,  it follows that $H$ is a simple  cycle. Relabel the vertices of the cycle so that $H$ is now a graph on $[m]$. Then $i=1$ since $1$ only has edges positively incident to it. Regardless of which vertex of $H$ is $j>1$,   there is a playable route $P$ starting at vertex $i$ and ending at $j$ such that $\phi(P)=v$.

\textbf{Case 2.}  $H$ has $1$ leaf. Then $H$ is a union of a simple cycle $C$ and a simple path $Q$. Relabel the vertices of $H$ so that it is a graph on the vertex set $[m]$. Let $l$ be the leftmost vertex of the cycle $C$ of $H$ and let $p$ be the vertex in common to $C$ and $Q$. Let $k$ be the unique leaf. 

If $l \neq p$, then $\{i, j\}=\{l, k\}$.  Thus, at least one of the edges of $C$ incident to $p$ are incident with an opposite sign to $p$ than the edge of $Q$ incident to $p$. Therefore, the edges on the path from $l$ to $p$ through the edge that is incident to $p$ in $C$ with the opposite sign to that of the edge of $Q$, and then the edges of path $Q$ form  a playable route $P$ such that   $\phi(P)=v$. 

If $l=p$ then we consider two possibilities, depending on whether  $l \not \in \{i, j\}$ or  $l \in \{i, j\}$. If  $l \not \in \{i, j\}$ then $i=k=1$ and $l\neq j$.  If $j \in C$, then  the edges of $Q$ (from $1$ to $l$) and the edges on the path from $l$ to $j$ through the edge that is incident to $j$ in $C$ with the sign of $e_j$ in $v$ make up a playable route $P$ with   $\phi(P)=v$.  If $j \in Q$ however, then, either the edges on the path from $i$ to $j$ along $Q$ make up a playable route $P$ with   $\phi(P)=v$, or the  the edges of $Q$ (from $1$ to $l$) and the edges of $C$ and then the edges on the path from $l$ to $j$  make up a playable route $P$ with   $\phi(P)=v$.

 If  $l=p$ and $l \in \{i, j\}$ then either $i=l$ or $j=l$. If $i=l$ then the edges on the path $Q$ from $l=1$ to $j=k$ make up a playable route $P$ with   $\phi(P)=v$. On the other hand if $j=l$ then $i=1$ and if the edge of $Q$ is incident to $l$ with the same sign as  that of the sign of $e_j$ in $v$, than   the edges of $Q$ make up a playable route $P$ with   $\phi(P)=v$. If, however, that sign is different, then it must be that $[e_j]v=1$ in which case all edges of $H$ (suitably ordered) make up a playable route $P$ with   $\phi(P)=v$. 

\textbf{Case 3.}   $H$ has $2$ leaves. Then $H$ could be  a path, or a union of a simple  cycle $C$ and two disjoint  paths $Q_1, Q_2$ attached to $C$ at vertices $p_1\neq p_2$, or  a union of a cycle $C$ and a tree $T$ with two leaves attached to $C$ at $t$. As in cases $1$ and $2$, in each case we can identify a playable route by inspection. We omit the details here. 
\qed

Proposition \ref{conn} yields a characterization of the vertices of $\mathcal{P}(G)$ for a connected 
$G\in  \mathcal{L}_n$.

 \begin{proposition} \label{phi}
  Let $G\in \mathcal{L}_n$. The map  $\phi$ defined by (\ref{phieq})  is a one-to-one  correspondence between playable routes and playable pairs  of $G$ and the vertices in $\overline{\mathcal{V}}_G$. 
 \end{proposition}
 
 \proof The proof is almost identical to that of Proposition \ref{conn}. The only difference is that the graph $H$ defined in the proof of  Proposition \ref{conn} could have two connected  components containing edges. The case of $H$   with one  connected  component containing edges is the same as in  the proof of  Proposition \ref{conn}.  
 
Let the  two  connected components of $H$  containing edges be $H_1$ and $H_2$. Then, $H_1$ and $H_2$ each  contributes exactly one coordinate with a nonzero coefficient, and thus each of them is  a union of a simple cycle (since $G\in \mathcal{L}_n$) and a possibly empty simple  path. The edges of $H_1$ and $H_2$, in a suitable order,  constitute  playable pairs.
 \qed

\begin{proposition} \label{vertices}
For  any graph $G$  the set of vertices  ${\overline{\mathcal{V}}}_G$ is the image  of   playable routes and pairs of $G$ under the map $\phi$ defined by (\ref{phieq}).
\end{proposition}

\proof Let  $P(G)={\rm ConvHull}(0, \v(i, j, \e) \mid \v(i, j, \e) \in {\mathcal{V}}_G)$, and let $\Delta$ be a central triangulation of $P(G)$. For each $\sigma \in \Delta$ we define $\mathcal{C}(\sigma)=\mathcal{C}({G'})$, where  the vertex set of $\sigma$ is $\{0, \v(i, j, \e) \mid (i,j,\e) \in G'\}$, $G'\subset G$ and $G' \in \mathcal{L}_n$. Then, $${\overline{\mathcal{V}}}_G \subset \mathcal{C}(G)=\bigcup_{\sigma \in \Delta} \mathcal{C}(\sigma).$$ Thus, any  $v \in{\overline{\mathcal{V}}}_G$ belongs to some ${\mathcal{C}}({G'})$. Therefore, $v\in{\overline{\mathcal{V}}}_{G'}$,  for $G' \in \mathcal{L}_n$,  $G'\subset G$.  By Proposition \ref{phi},  there is a playable route $P$ or pair $(P_1, P_2)$ in $G'$, such that $v=\phi(P)$ or $v=\phi(P_1, P_2)$. But all playable routes and pairs of  $G'$ are also playable routes and pairs of  $G$. 
\qed

Propositions \ref{conn}, \ref{phi} and \ref{vertices} imply Proposition \ref{playable_route}.

\section{ The proof of the Reduction Lemma}
\label{sec:reduction_lemma}

This section is devoted to proving the Reduction Lemma (Lemma \ref{reduction_lemma}).  As we shall see in Section \ref{sec:forest}, the Reduction  Lemma is the ``secret force"  that makes   everything fall into its place for coned root polytopes. We start by   characterizing the root polytopes which are simplices, then in  Lemma \ref{equivalent} we prove that equations (\ref{eq1}) and (\ref{eq2}) are equivalent definitions for the root polytope $\mathcal{P}(G)$,  and finally we prove the Cone Reduction Lemma (Lemma \ref{cone_reduction_lemma}), which, together with Lemma \ref{equivalent} implies the Reduction Lemma.

  \begin{lemma} \label{simplex}  
 For a  graph $G$ on the vertex set $[n]$ with  $d$ edges, the polytope $\mathcal{P}(G)$ as defined by (\ref{eq1}) is a simplex if and only if $G$ is alternating and $G \in \mathcal{L}_n$. 
  \end{lemma}
                   
   \proof
   It follows from equation (\ref{eq1}) that for a minimal graph $G$ the polytope  $\mathcal{P}(G)$ is a simplex if and only if  the vectors corresponding to the edges of $G$ are linearly independent and  $\mathcal{C}(G) \cap \Phi^+=\mathcal{V}_G$.  
   
  The vectors corresponding to the edges of $G$ are linearly independent if and only if  $G \in \mathcal{L}_n$. 
  By Proposition \ref{playable_route},   $\mathcal{C}(G) \cap \Phi^+=\mathcal{V}_G$  if and only if  $G$ contains no edges incident to a vertex $v \in V(G)$ with opposite signs,  i.e. $G$ is alternating.   
   \qed

\begin{lemma} \label{equivalent}
For any  graph $G$ on the vertex set $[n]$,  
 
 $$\emph{ConvHull}(0,   \v(i, j, \e)  \mid  (i, j, \e)  \in \overline{G})=\mathcal{P}(C_n^+) \cap \mathcal{C}(G).$$ 
\end{lemma}

\proof For a graph $H$ on the vertex set $[n]$, let $\sigma(H)=\textrm{ConvHull}(0, \v(i, j, \e)  \mid  (i, j, \e)  \in   H)$. Then, $\sigma(\overline{G} )=\textrm{ConvHull}(0, \v(i, j, \e)  \mid  (i, j, \e)  \in \overline{G})$. Let $\sigma(\overline{G} )$ be a $d$-dimensional polytope for some $d \leq n$ and consider any central triangulation of it:  $\sigma(\overline{G} )=\cup_{F \in \mathcal{F}}\sigma(F)$, where $\{\sigma(F)\}_{F \in \mathcal{F}}$ is a set of $d$-dimensional simplices with disjoint interiors, $E(F) \subset E(\overline{G})$, $F \in \mathcal{F}$.  Since $\sigma(\overline{G} )=\cup_{F \in \mathcal{F}}\sigma(F)$ is a central triangulation, it follows that $\sigma(F)=\sigma(\overline{G} ) \cap \mathcal{C}(F)$, for $F \in \mathcal{F}$, and $\mathcal{C}(G)=\cup_{F \in \mathcal{F}}  \mathcal{C}(F)$.

Since $\sigma(F)$, $F \in \mathcal{F}$, is a  $d$-dimensional simplex, it follows that $F\in \mathcal{L}_n$ and has  $d$ edges. Furthermore, $F \in \mathcal{F}$ is  alternating, as otherwise there are edges $(i, j, \e_1), (j, k, \e_2) \in E(F) \subset E(\overline{G})$ incident to $j$ with opposite signs, and while $\v(i, j, \e_1)+\v (j, k, \e_2) \in \sigma(\overline{G} ) \cap \mathcal{C}(F)$, $\v(i, j, \e_1)+\v (j, k, \e_2) \not \in \sigma(F)$, contradicting that  $\cup_{F \in \mathcal{F}}\sigma(F)$ is a central triangulation of $\sigma(\overline{G} )$. Thus,  $\overline{F}=F$, and  $\sigma(F)=\sigma(\overline{F})$. It is clear that  
 $\sigma(\overline{F})=\textrm{ConvHull}(0, \v(i, j, \e)  \mid  (i, j, \e)  \in \overline{F})\subset \mathcal{P}(C_n^+) \cap \mathcal{C}(F),$ $F \in \mathcal{F}$. Since if $x=(x_1, \ldots, x_{n+1})$ is in the facet of   $\sigma(\overline{F})$ opposite the origin, then  $|x_1|+\cdots+|x_{n+1}|=2$ and for any point  $x=(x_1, \ldots, x_{n+1}) \in \mathcal{P}(C_n^+) $, $|x_1|+\cdots+|x_{n+1}|\leq 2$ it follows that  $\mathcal{P}(C_n^+) \cap \mathcal{C}(F)\subset \sigma(\overline{F})$. Thus, $\sigma(\overline{F})=\mathcal{P}(C_n^+) \cap \mathcal{C}(F)$.  Finally, $\textrm{ConvHull}(0, \v(i, j, \e)  \mid  (i, j, \e)  \in \overline{G})=\sigma(\overline{G} )=\cup_{F \in \mathcal{F}}\sigma(F)=\cup_{F \in \mathcal{F}}\sigma(\overline{F})
=\cup_{F \in \mathcal{F}}(\mathcal{P}(C_n^+) \cap \mathcal{C}(F))=\mathcal{P}(C_n^+) \cap (\cup_{F \in \mathcal{F}}\mathcal{C}(F)) = \mathcal{P}(C_n^+) \cap \mathcal{C}(G)$ as desired.

\qed

    \begin{lemma} \label{cone_reduction_lemma} \textbf{(Cone Reduction Lemma)} 
Given   a graph $G_0 \in \mathcal{L}_n$ with $d$ edges, let   $G_1, G_2, G_3$ be the graphs described  by  any one of the equations (\ref{graphs1})-(\ref{graphs6}).   Then  $G_1, G_2, G_3 \in \mathcal{L}_n$,
$$\mathcal{C}(G_0)=\mathcal{C}(G_1) \cup \mathcal{C}(G_2)$$   where all cones  $\mathcal{C}(G_0), \mathcal{C}(G_1), \mathcal{C}(G_2)$ are   $d$-dimensional and    
$$\mathcal{C}(G_3)=\mathcal{C}(G_1) \cap \mathcal{C}(G_2)  \mbox{   is $(d-1)$-dimensional. } $$
\end{lemma}

    The proof of Lemma  \ref{cone_reduction_lemma} is the same as that of the Cone Reduction Lemma in the type $A_n$ case; see \cite[Lemma 7]{kar}.
\medskip
  
         \noindent \textit{Proof of the  Reduction Lemma (Lemma \ref{reduction_lemma}).} Straightforward corollary of Lemmas \ref{equivalent} and \ref{cone_reduction_lemma}.\qed

\medskip

    \section{Volumes of root polytopes and the number of monomials in reduced forms}                   
                       \label{sec:forest}
                       
                       In this section we use the Reduction Lemma to establish the link between the volumes of  root polytopes and the number of monomials in reduced forms. In fact we shall see that if we know either of these quantities, we also know the other.

            \begin{proposition} \label{2/n!} Let $G_0 \in \mathcal{L}_n$ be a connected graph on the vertex set $[n]$ with $n$ edges, and let $\mathcal{T}^\mathcal{S}$ be an $\mathcal{S}$-reduction tree with root labeled  $G_0$. Then,  $$\vol_n(\mathcal{P}(G_0))=\frac{2 f(G_0)}{n!},$$ where $f(G_0)$ denotes the number of   leaves of $\mathcal{T}^\mathcal{S}$ labeled by graphs with $n$ edges.
\end{proposition}

\begin{proof} By the Reduction Lemma (Lemma \ref{reduction_lemma}) $\vol_n(\mathcal{P}(G_0))=\sum_G \vol_n(\mathcal{P}(G))$, where $G$ runs over the  leaves of $\mathcal{T}^\mathcal{S}$ labeled by graphs with $n$ edges. We now prove that for each $G$ with $n$ edges  labeling a leaf  of $\mathcal{T}^\mathcal{S}$  with root labeled  $G_0$, $\vol_n(\mathcal{P}(G))=\frac{2}{n!}$. Since $G_0 \in \mathcal{L}_n$ is a connected graph on the vertex set $[n]$ with $n$ edges, so are all its successors with $n$ edges. If $G$ labels a leaf of $\mathcal{T}^\mathcal{S}$, then $G$ satisfies the conditions of Lemma \ref{simplex}. Thus, $\mathcal{P}(G)$ is a simplex.

The volume of  $\mathcal{P}(G)$  can be calculated by calculating the determinant $\det(M)$ of the matrix $M$ whose rows are the vectors $\v(e)$,  $e \in E(G)$, written in the standard basis.  If $v \in [n]$ is a vertex of degree $1$ in $G$, the $v^{th}$ column  contains a single $1$ or $-1$ in the row corresponding to the edge incident to $v$. Let this row be the ${v_r}^{th}$.  Delete the   $v^{th}$ column and $v_r^{th}$ row from $M$ and delete the edge incident to $v$ in $G$ obtaining a new graph. Successively identify the leaves in the new graphs and delete the corresponding columns and rows from their matrices until we obtain a graph $C$ that is a simple cycle and the corresponding  matrix $M'$. The rows of $M'$ are the vectors $\v(e)$,  $e \in E(C)$.  By Laplace expansion, $|\det(M)|=|\det(M')|$.  Since $G \in \mathcal{L}_n$, so is $C \in \mathcal{L}_n$.  Thus,  $\det(M')\neq 0$. Expand $M'$ by any of its rows obtaining matrices $M_1$ and $M_2$. Then we get $|\det(M')|=|\det(M_1)|+|\det(M_2)|=2$, since both $M_1$ and $M_2$ are such that their entries are all $0, 1$ or $-1$, each row (column) except one has exactly two nonzero entries, and the remaining one exactly one nonzero entry. 
Thus, $\vol_n(\mathcal{P}(G))=\det(M)/n!=2/n!$.

\end{proof}

 A general version of Proposition \ref{2/n!} can be proved for any connected $G_0 \in \mathcal{L}_n$ using the following lemma.
 
      \begin{lemma} \label{2^k/d!} Let $G\in \mathcal{L}_n$ be an alternating  graph on the vertex set $[n]$ with $d$ edges, with $c$ connected components of which   $k\leq c$  contain simple cycles. Then,  $$\vol_d(\mathcal{P}(G))=\frac{2^k}{d!}.$$  
 \end{lemma}

\proof Let $M^a$   be the matrix whose rows are the vectors $\v(i, j, \e)$,  $(i, j, \e) \in E(G)$, written in the standard basis. Matrix $M^a$ is a $d \times n $ matrix. The rows and columns of $M^a$ can be rearraged so that it has a block form in which the blocks $B_1, \ldots, B_c$ on the diagonal correspond to the connected components of $G$, while all other blocks are $0$.  Since $G\in \mathcal{L}_n$ satisfies the conditions of Lemma \ref{simplex},   $\mathcal{P}(G)$ is a simplex, $\vol_{d}(\mathcal{P}(G))\neq 0$ and $\vol_{d}(\mathcal{P}(G))$ can be calculated by dropping some $n-d$ columns of $M^a$  such that the resulting matrix $M$ has nonzero determinant. Then, $\vol_d(\mathcal{P}(G))=|\det(M)|/d!$. Drop a column $b_i$  from the block matrix $B_i$ if the block $B_i$ corresponds to a tree on $m$ vertices, obtaining  matrix $B_i'$ with nonzero determinant. Then, $|\det(B_i')|=1$. If $B_i$  corresponds to a connected component  of $G_0$ with $m$ vertices and $m$ edges, then  $B_i'=B_i$ and  $|\det(B_i)|=2$. Since there are $n-d$ connected components which are trees, if we drop the columns $b_i$ from $M^a$ for all blocks $B_i$  corresponding to a tree obtaining a matrix $M$, then $\vol_d(\mathcal{P}(G))=\frac{|\det(M)|}{d!}$. Since $M$ has a special block form with blocks $B_i'$ along diagonal and zeros otherwise, we have that $|\det(M)|=| \prod_{i=1}^c det(B_i')|=2^k$.

\qed
 
 \begin{proposition} \label{2^k}
  Let $G_0 \in \mathcal{L}_n$ be a   graph on the vertex set $[n]$ with $d$ edges, with $c$ connected components of which   $k\leq c$  contain  cycles.  Let $\mathcal{T}^\mathcal{S}$ be an $\mathcal{S}$-reduction tree with root labeled  $G_0$. Then,  $$\vol_d(\mathcal{P}(G_0))=\frac{2^k f(G_0)}{d!},$$ where $f(G_0)$ denotes the number of   leaves of $\mathcal{T}^\mathcal{S}$ labeled by graphs with $d$ edges.
 \end{proposition}
 
 The proof of Proposition \ref{2^k} proceeds analogously to Proposition \ref{2/n!}, in view of Lemma \ref{2^k/d!}.

      \begin{corollary}
Let $G_0\in \mathcal{L}_n$  and let $m^\mathcal{S}[G_0]$ be the monomial corresponding to it. Then for any reduced form $P_n^\mathcal{S}$  of $m^\mathcal{S}[G_0]$, the value of $P_n^\mathcal{S}(x_{ij}=y_{ij}=z_i=1, \beta=0)$ is independent of the order of  reductions performed. 
\end{corollary}

\proof Note that $P_n^\mathcal{S}(x_{ij}=y_{ij}=1, \beta=0)=f(G_0),$ as defined in Proposition \ref{2^k}. Since $\vol_d(\mathcal{P}(G_0))$ is only dependent on $G_0$,  the value of $P_n^\mathcal{S}(x_{ij}=y_{ij}=z_i=1, \beta=0)$ is independent of the particular reductions performed. 
\qed

\medskip

With analogous methods the following proposition about reduced forms in $\c$ can also be proved. 
     
      \begin{proposition} \label{ce}
Let $G_0\in \mathcal{L}_n$  and let $m^\mathcal{S}[G_0]=m^{\mathcal{B}^c}[G_0]$ be the monomial corresponding to it. Then for any reduced form $P_n^{\mathcal{B}^c}$ of $m^\mathcal{S}[G_0]$ in $\c$, the value of $P_n^{\mathcal{B}^c}(x_{ij}=y_{ij}=z_i=1)$ is independent of the order of  reductions performed. 
\end{proposition}

    \section{Reductions in the noncommutative case}
\label{reductionsB}  

In this section we turn our attention to the noncommutative algebra $\b$. We consider  reduced forms of monomials  in $\b$ and the reduction rules correspond to the relations $(5)-(9')$ of $\b$:

\medskip

(5) $x_{ij}x_{jk}\rightarrow x_{ik}x_{ij}+x_{jk}x_{ik}$,  for  $1\leq i<j<k\leq n$,

\medskip

($5'$) $x_{jk}x_{ij}\rightarrow x_{ij}x_{ik}+x_{ik}x_{jk}$, for  $1\leq i<j<k\leq n$,

\medskip

(6) $x_{ij}y_{jk}\rightarrow  y_{ik}x_{ij}+y_{jk}y_{ik}$,  for  $1\leq i<j<k\leq n$,

\medskip

($6'$) $y_{jk}x_{ij} \rightarrow x_{ij}y_{ik}+y_{ik}y_{jk}$, for  $1\leq i<j<k\leq n$,

\medskip

(7) $x_{ik}y_{jk} \rightarrow y_{jk}y_{ij}+y_{ij}x_{ik}$, for  $1\leq i<j<k\leq n$,

\medskip

($7'$) $y_{jk}x_{ik} \rightarrow y_{ij}y_{jk}+x_{ik}y_{ij}$, for  $1\leq i<j<k\leq n$,

\medskip

(8) $y_{ik}x_{jk}\rightarrow x_{jk}y_{ij}+y_{ij}y_{ik}$, for  $1\leq i<j<k\leq n$,

\medskip

($8'$) $x_{jk}y_{ik} \rightarrow y_{ij}x_{jk}+y_{ik}y_{ij}$, for  $1\leq i<j<k\leq n$,

\medskip

(9) $x_{ij}z_j \rightarrow z_i x_{ij}+ y_{ij} z_i + z_j y_{ij}$, for $i<j$

\medskip

($9'$) $z_jx_{ij} \rightarrow x_{ij}z_i+  z_i y_{ij}+  y_{ij}z_j$, for $i<j$

\medskip

As observed in Proposition \ref{ce}, in the commutative counterpart of $\b$, $\c$,  the number of monomials in a reduced form of $w_{C_n}$ is the same, regardless of the order of the reductions performed. In this section we develop  the tools necessary for proving the uniqueness of the reduced form in $\b$ for $w_{C_n}$ and other monomials. The key concept is that of a ``good" graph, which property is preserved under the reductions.

As in the commutative case before, we can to phrase the reduction process in terms of graphs.   Let    $m=\prod_{l=1}^p \w(i_l, j_l,\e_l)$ be a monomial  in variables $x_{ij}, y_{ij}, z_k, 1 \leq i<j \leq n, k \in [n]$, where $\w(i, j, -)=x_{ij}$ for $i<j$, $\w(i, j, -)=x_{ji}$ for $i>j$, $\w(i, j,+)=y_{ij}$ and $\w(i, i,+)=z_{i}$. We can think of $m$ as a graph $G$ on the vertex set $[n]$ with $p$ edges labeled $1, \ldots, p$, such that the edge labeled $l$ is $(i_l, j_l, \e_l)$.  Let $G^\mathcal{B}[m]$ denote the  edge-labeled graph just described.
  Let $(i, j, \e)_a$  denote an   edge  $(i , j, \e )$ labeled $a$. Recall  that in our  edge notation $(i, j, \e)=(j, i, \e)$, i.e., vertex-label  $i$ might be smaller or greater than $j$. We can reverse the process and   obtain a monomial from an edge labeled graph $G$. 
  Namely, if $G$ is   edge-labeled with labels $1, \ldots, p$, we can also associate to it the noncommutative monomial $m^\mathcal{B}[G]=\prod_{a=1}^p \w(i_a, j_a,\e_a)$, where $E(G)=\{ (i_a, j_a, \e_a)_a \mid a \in [p]\}$.

In terms of graphs the partial  commutativity of $\b$, as described by relations (2)-(4),    means that if $G$ contains two edges $(i, j, \e_1)_a$ and $(k, l, \e_2)_{a+1}$ with $i, j, k, l$ distinct, then we can replace these edges by $(i, j, \e_1)_{a+1}$ and $(k, l, \e_2)_{a}$, and vice versa. For illustrative purposes we write out the graph reduction for relation (5) of $\b$. If there are two edges $(i, j, -)_a$ and $(j, k, -)_{a+1}$ in $G_0$, $i<j<k$,  then we replace $G_0$ with two graphs $G_1, G_2$ on the vertex set $[n]$ and edge sets 

\begin{eqnarray} \label{graphs}
E(G_1)&=&E(G_0)\backslash \{(i, j, -)_a\}\backslash \{(j, k, -)_{a+1}\}\cup \{(i, k, -)_a\}\cup \{(i, j, -)_{a+1}\} \nonumber \\
E(G_2)&=&E(G_0)\backslash \{(i, j, -)_a\}\backslash \{(j, k, -)_{a+1}\}\cup \{(j, k, -)_a\}\cup \{(i, k, -)_{a+1}\}  \nonumber  
\end{eqnarray}

 Relations $(5')-(9')$ of $\b$  can be translated into graph language analogously. We say that $G_0$ reduces to $G_1$ and $G_2$ under reductions $(5)-(9')$.

While in the commutative case   reductions on $G^\mathcal{S}[m]$ could result in crossing graphs, we prove that in $\b$ all   reductions   preserve the  noncrossing nature of graphs, provided that we started with a    suitable noncrossing graph $G$. A graph $G$ is {\bf noncrossing} if    there are no vertices $i <j<k<l$ such that $(i, k, \e_1)$ and $(j, l, \e_2)$ are edges of $G$. We also show that under reasonable  circumstances,   if  in $\c$ a reduction could be applied to edges  $e_1$ and $e_2$, then   after suitably many allowed commutations in $\b$  it is possible to perform  a reduction on $e_1$ and $e_2$ in $\b$.  

We now define two central notions of the noncommutative case, that of a well-structured graph and that of a well-labeled graph. 
\medskip

A graph  $H$ on the vertex set $[n]$  is \textbf{well-structured} if it satisfies the following conditions:

$(i)$ $H$ is noncrossing. 

$(ii)$ For any two edges $(i, j, +), (k, l, +) \in H$, $i<j, k<l$, it must be that $i<l$ and $k<j$.

$(iii)$ For any two edges $(i, i, +), (k, l, +) \in H$, $k<l$, it must be that $k\leq i \leq l$ .


$(iv)$ There are no edges $(i, i, +), (k, j, -) \in H$ with $k<i<j$.

$(v)$ There are no edges $ (i, j, +), (k, l, -) \in H$ with $k\leq i<j \leq l$.



$(vi)$ Graph $H$ is connected, contains exactly one loop, and contains no nonloop cycles.

\medskip

Condition $(vi)$ implies that any well-structured graph on the vertex set $[n]$ contains $n$ edges. 

\medskip

A graph  $H$ on the vertex set $[n]$ and $p$ edges labeled   $1, \ldots, p$  is \textbf{well-labeled} if it satisfies the following conditions:

 $(i)$ If edges $(i, j, \e_1)_a$ and $(j, k, \e_2)_b$ are in $H$, $i<j<k$, $\e_1, \e_2 \in \{-, +\}$,  then $a<b$.

$(ii)$ If edges $(i, j, \e_1)_a$ and $(i, k, \e_2)_b$ in $H$ are such that $i<j < k$, $\e_1, \e_2 \in \{-, +\}$,  then $a>b$.

$(iii)$  If edges $(i, j, \e_1)_a$ and $(k, j, \e_2)_b$ in $H$ are such that $i < k<j$, $\e_1, \e_2 \in \{-, +\}$,  then $a>b$.

$(iv)$  If edges $(i, i, +)_a$ and $(i, j, -)_b$ in $H$ are such that $i <j$,  then $a<b$.

$(v)$  If edges $(j, j, +)_a$ and $(i, j, -)_b$ in $H$ are such that $i <j$,  then $a>b$.

$(vi)$  If edges $(i, i, +)_a$ and $(i, j, +)_b$ in $H$ are such that $i <j$,  then $a>b$.

$(vii)$  If edges $(j, j, +)_a$ and $(i, j, +)_b$ in $H$ are such that $i <j$,  then $a<b$.

\medskip

Note that no graph $H$ with a nonloop cycle can be well-labeled. However,  every well-structured graph   can be well-labeled. We call graphs that are both well-structured and  well-labeled  {\bf good} graphs.

A \textbf{$\mathcal{B}$-reduction tree}  $\mathcal{T}^\mathcal{B}$ is defined analogously to an $\mathcal{S}$-reduction tree, except we use the noncommutative reductions to describe the children. See Figure \ref{fig:eps} for an example.
 A graph $H $ is called a \textbf{$\mathcal{B}$-successor} of $G$ if it is obtained by a series of reductions from $G$.    

    \begin{figure}[htbp] 
\begin{center} 
\includegraphics[width=1.2\textwidth]{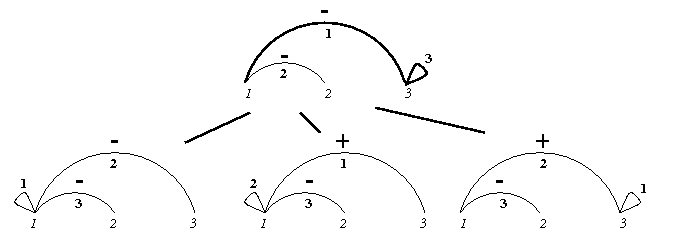} 
\caption{A $\mathcal{B}$-reduction tree with root corresponding to the monomial $x_{13}x_{12}z_3$. Note that in order to perform a reduction on this monomial we commute variables $x_{13}$ and $x_{12}$. In the $\mathcal{B}$-reduction tree we only record the reductions, not the commutations.  Summing the monomials corresponding to  the graphs labeling the  leaves  of the reduction tree    we obtain a reduced form $P^\mathcal{B}_n$ of  $x_{13}x_{12}z_3$,  $P^\mathcal{B}_n=z_1x_{13}x_{12}+y_{13}z_1x_{12}+  z_3y_{13}x_{12}$. } 
\label{fig:eps}
\end{center} 
\end{figure} 

\begin{lemma} \label{huh}
If the root of a $\mathcal{B}$-reduction  tree is labeled by   a good graph, then all nodes of it  are  also labeled by  good graphs.  
\end{lemma}

The proof of Lemma \ref{huh} is an analysis of the local changes that happen during the noncommutative reduction process. An analogous lemma for  type $A_n$ is proved in \cite[Lemma 12]{kar}.

 A reduction applied to a noncrossing graph $G$  is \textbf{noncrossing} if the graphs resulting from the reduction are also noncrossing.

The following is then an immediate corollary of Lemma \ref{huh}.

\begin{corollary}
If $G$ is a  good graph, then all reductions that can be applied to $G$ and its $\mathcal{B}$-successors   are  noncrossing.
\end{corollary}

Let $e_1=(i_1, j_1, \e_1)_{a_1}, e_2=(i_2, j_2, \e_2)_{a_2}, e_3=(i_3, j_3, \e_3)_{a_3}$ be edges of the graph  $H$ such that 
 in the commutative algebra $\c$ a reduction could be performed on $e_1$ and $e_2$ as well as on $e_1$ and $e_3$. Suppose  that $a_1<a_2<a_3$. Then we say,  in the noncommutative case $\b$, that performing reduction on edges $e_1$ and $e_2$  is a {\bf priority} over performing reduction 
on edges $e_1$ and $e_3$.  We give a few concrete examples of this priority below. 

 \begin{example}  Performing reduction (6) on edges $(i, j, -), (j, k, +) \in H$, $i<j<k$,  is a {\bf priority} over performing reduction (9) on edges $(i, j, -), (j, j, +) \in H$. Performing reduction (9) on edges $(i, j, -), (j, j, +) \in H$,   is a {\bf priority} over performing reduction (5) on edges $(i, j, -), (j, k, -) \in H$, $i<j<k$.   Performing reduction (9) on edges $(i, j, -), (j, j, +) \in H$,  is a {\bf priority} over performing reduction (9) on edges $(k, j, -), (j, j, +) \in H$, $i<k<j$.  Performing reduction (9) on edges $(i, j, -), (j, j, +) \in H$ is a {\bf priority} over performing reduction (8) on edges $(i, j, -), (k, j, +) \in H$, $k<i<j$. 
\end{example}

\begin{lemma} \label{cross}
Let $G$ be a   good graph.  Let $e_1$ and $e_2$ be edges of $G$ such that one of the reductions $(5)-(9')$ could be applied to them in the commutative case, and such that the reduction would be noncrossing. 
 Then after finitely many applications of allowed commutations in $\b$  we can perform a   reduction on  edges  $e_1$ and $e_2$, provided there is no edge $e_3$ in the graph  such that reducing $e_1$ and $e_3$ or $e_2$ and $e_3$  is a priority over reducing $e_1$ and $e_2$.

\end{lemma}

The proof of Lemma \ref{cross} proceeds by inspection. An analogous lemma for  type $A_n$ is proved in \cite[Lemma 14]{kar}.
  
\section{The Proof  of Kirillov's Conjecture}
\label{reductionsB1}

In this section we prove Conjecture 1, construct a triangulation of $\mathcal{P}(C_n^+)$ and compute its volume.  In order to do this we 
study alternating well-structured graphs. Recall that  
an alternating well-structured graph $T^l$ is the union of a 
  noncrossing  alternating tree $T$ on the vertex set $[n]$  and a loop, that is,   $T^l=([n], E(T) \cup \{(k, k, +)\}),$ for some $k \in [n]$ for which  $T^l$ is alternating. A  well-labeling that  will play a special role in this section is the lexicographic labeling, defined below. 
  
  The \textbf{lexicographic order} on the edges of a  graph $G$ with $m$ edges is as follows. Edge $(i_1, j_1, \e)$ is less than edge $(i_2, j_2, \e)$, $\e \in \{+, -\}$, in the  lexicographic  order  if  $j_1>j_2$, or $j_1=j_2$ and $i_1>i_2$. Furthermore, any positive edges is less than any  negative edges in the lexicographic ordering. Graph  $G$ is said to have \textbf{lexicographic  edge-labels} if its edges are labeled by integers $1, \ldots, m$ such that  if   edge $(i_1, j_1, \e_1)$ is less than edge $(i_2, j_2, \e_2)$ in lexicographic  order, then the label of $(i_1, j_1, \e_1)$ is less than the label of   $(i_2, j_2, \e_2)$ in the usual order on the integers. Given any graph $G$ there is a unique edge-labeling of it which is lexicographic. Note that our definition of lexicographic is closely related to the conventional definition, but it is not the same. For an example of lexicographic  edge-labels, see the graphs labeling the leaves of the $\mathcal{B}$-reduction tree in Figure \ref{fig:eps}.

\begin{lemma} \label{labeling}
If  $T^l$ is an alternating good graph, then upon some number of commutations performed on $T^l$,  it is possible to obtain  ${T}^l_1$ with lexicographic  edge-labels.
\end{lemma}

\begin{proof}
If edges  $e_1$ and $e_2$ of $T^l$   share a vertex and if $e_1$ is less than $e_2$ in the lexicographic order, then the label of $e_1$ is less than the label of $e_2$ in the usual order on integers by the definition of well-labeling on alternating well-structured graphs. Since commutation  swaps the labels of two vertex disjoint edges labeled by consecutive integers in  a graph, these swaps do not affect the relative order of the labels on edges sharing vertices.  Continue these swaps until the lexicographic order is obtained. 
\end{proof}

\begin{proposition} \label{non}
By choosing the series of reductions suitably, the set of leaves of a $\mathcal{B}$-reduction tree with root labeled by  $G^\mathcal{B}[w_{C_n}]$ can be all  alternating well-structured graphs $T^l$ on the vertex set  $[n]$     with lexicographic  edge-labels.  The number of such graphs is ${{2n-1}\choose n}$.
\end{proposition}

\proof 
By the correspondence between the leaves of a $\mathcal{B}$-reduction tree    and simplices in a subdivision of  $\mathcal{P}(G^\mathcal{B}[w_{C_n}])$ obtained from the Reduction Lemma (Lemma \ref{reduction_lemma}), it follows that no graph with edge labels disregarded appears more than once among the leaves of a $\mathcal{B}$-reduction tree. Thus,   it suffices to prove that any  alternating well-structured graph $T^l$ on the vertex set  $[n]$ appears among the leaves of a $\mathcal{B}$-reduction tree and that all these graphs have lexicographic  edge-labels. 

First perform all possible reductions on the graph and its successors not involving the loop $(n, n, +)$. According to \cite[Theorem 18]{kar} the outcome is all noncrossing alternating spanning trees with lexicographic ordering on the vertex set $[n]$  and edge $(1, n, -)$ present. Let $T_1, \ldots T_w$ be the trees just described and $T_i^l=([n], E(T_i) \cup \{(n, n, +)\})$, $i \in [w]$. It is clear from the definition of reductions that the only edges involved in further reducing $T_i^l$, $i \in [w]$ are the ones incident to vertex $n$. Thus, in order to understand what the leaves of a reduction tree with root labeled $T^l_i$, $i \in [w]$, are, it suffices to understand the leaves of a reduction tree with root labeled $G=([k+1], \{(k+1, k+1, +), (i, k+1, -) \mid i \in [k])$, $k \in \{1, 2, \ldots, n-1\}$. It follows by inspection that the  leaves of a reduction tree with root labeled $G$ are of the form $([k+1], E(G_1) \cup E(G_2))$, where $G_1$ is a connected well-structured graph with only positive edges (having exactly one  loop) on $[l]$, $l \in [k+1]$, of which there are $2^{l-1}$ and $G_2=([k+1], \{ (i, k+1) \mid i \in \{l, l+1, \ldots, k\})$. It follows that all   alternating well-structured graphs $T^l$ are among the leaves of the particular $\mathcal{B}$-reduction tree described. Since all these graphs are well-labeled, having started with a good graph, by Lemma \ref{labeling} we can assume they have lexicographic  edge-labels.

From the description of the reductions above it is clear that the number of leaves of this particular reduction tree is $$\sum_{k=1}^{n-1} T(n, k) \cdot (2^{k+1}-1),$$ where   
 $$T(n, k)={{2n-k-3} \choose {n-k-1}} \frac{k}{n-1}$$ is the number of noncrossing alternating  trees on the vertex set $[n]$  with exactly $k$  edges incident to $n$, and  $2^{k+1}-1$ is the number  of leaves of the  reduction tree with root labeled $G([k+1], \{(k+1, k+1, +), (i, k+1, -) \mid i \in [k])$ as above. The formula for $T(n, k)$ follows by   a simple bijection between noncrossing alternating  trees on the vertex set $[n]$  with exactly $k$  edges incident to $n$ and ordered trees on the vertex set $[n]$ with the root having degree $k$. By  equations (6.21), (6.22), (6.28) and the bijection presented in Appendix E.1. in \cite{dyck}, ordered trees on the vertex set $[n]$ with the root having degree $k$ are enumerated by $T(n, k)$.   
 Since $\sum_{k=1}^{n-1} T(n, k) \cdot (2^{k+1}-1)={{2n-1}\choose n},$ the proof is complete.

\qed

       \begin{theorem} \label{ajaj}
The set of leaves of a $\mathcal{B}$-reduction tree  with root labeled by  $G^\mathcal{B}[w_{C_n}]$ is, up to commutations,  the set of all   alternating well-structured graphs    on the vertex set $[n]$  with  lexicographic  edge-labels.
\end{theorem}

\proof

By Proposition \ref{non} there exists a $\mathcal{B}$-reduction tree which satisfies the conditions above. By Proposition \ref{2^k}  the number of graphs with $n$ of  edges among the leaves of an  $\mathcal{S}$-reduction tree is independent of the particular $\mathcal{S}$-reduction tree, and, thus, the same is true for a $\mathcal{B}$-reduction tree.  Since all graphs labeling the leaves of a $\mathcal{B}$-reduction tree  with root labeled by $G^\mathcal{B}[w_{C_n}]$ have to be good by Lemma \ref{huh}, and   no graph, with edge-labels disregarded, can appear twice among the leaves of a $\mathcal{B}$-reduction tree, imply, together with Lemma \ref{labeling},  the statement of Theorem \ref{ajaj}.
\qed

\medskip 

As corollaries of  Theorem \ref{ajaj} we obtain  the  characterziation of reduced forms of the noncommutative monomial $w_{C_n}$, a triangulation of $\mathcal{P}(C_n^+)$ and a way to compute its volume. 

\begin{theorem} \label{main}
If  the polynomial $P^\mathcal{B}_n (x_{ij}, y_{ij}, z_{i})$ is a reduced form of  $w_{C_n}$,  then up to commutations

$$P^\mathcal{B}_n(x_{ij}, y_{ij}, z_{i})=\sum_{T^l} m^\mathcal{B}[T^l],$$

\noindent where the sum runs over all   alternating well-structured graphs $T^l$   on the vertex set $[n]$ with lexicographic  edge-labels.\end{theorem}

  \begin{theorem} \label{main2}
If  the polynomial $P^{\mathcal{B}^c}_n(x_{ij}, y_{ij}, z_{i})$ is a reduced form of  $w_{C_n}$ in $\c$,  then

$$P^{\mathcal{B}^c}_n(x_{ij}=y_{ij}=z_{i}=1)= {{2n-1}\choose n}.$$
\end{theorem}

\proof Proposition \ref{ce} and Theorem \ref{main} imply $P^{\mathcal{B}^c}_n(x_{ij}=y_{ij}=z_{i}=1)= {{2n-1}\choose n}.$ \qed

\begin{theorem} \label{main3}
Let $T^l_1, \ldots, T^l_m$ be all  alternating well-structured graphs on the vertex set $[n]$. Then   $\mathcal{P}(T^l_1), \ldots, \mathcal{P}(T^l_m)$  are $n$-dimensional simplices forming a triangulation of $\mathcal{P}(C_n^+)$. Furthermore, 
$$\vol_n(\mathcal{P}(C_n^+))={{2n-1}\choose n}\frac{2}{n!}.$$
\end{theorem}

\proof
 The Reduction Lemma implies the first claim, and  Proposition \ref{2/n!} implies  $\vol_n(\mathcal{P}(C_n^+))={{2n-1}\choose n}\frac{2}{n!}.$
\qed

The value of  the volume of $\mathcal{P}(C_n^+)$ has previously been  observed by Fong  \cite[p. 55]{fong}. 
 
 \section{The general case}
  \label{sec:gen} 
  
  In this section we find   analogues of Theorems \ref{ajaj}, \ref{main}, \ref{main2} and \ref{main3} for any well-structured graph $T^l$ on the vertex set $[n]$.  
  
\begin{proposition} \label{gen:non} Let $T^l$ be a well-structured graph on the vertex set $[n]$.  By choosing the series of reductions suitably, the set of leaves of a $\mathcal{B}$-reduction tree with root labeled by $T^l$ can be all 
 alternating well-structured spanning graphs $G$ of  $\overline{T^l}$  on the vertex set $[n]$ with lexicographic  edge-labels. 
 \end{proposition}
 
  \begin{proof}  
All   graphs labeling the leaves of a  $\mathcal{B}$-reduction tree  must be   alternating well-structured spanning graphs $G$ of  $\overline{T^l}$. Also, it is possible to obtain any  well-structured graph $T^l$ on the vertex set $[n]$   as an $\mathcal{B}$-successor of $P^l$. Furthermore, if $T^l$ and $T_1^l$ are two $\mathcal{B}$-successor of $P^l$ in the same  $\mathcal{B}$-reduction tree, and neither is the $\mathcal{B}$-successor of the other, then the intersection of $\overline{T^l}$ and $\overline{T_1^l}$ does not contain a well-structured graph $G$, as the existence of such a graph  would imply that  $\mathcal{P}(T^l)$ and $\mathcal{P}(T_1^l)$ have a common interior point, contrary to the Reduction Lemma.  Since the set of leaves of a $\mathcal{B}$-reduction tree with root labeled by  $P^l$ is, up to commutations,  the set of all   alternating well-structured graphs    on the vertex set $[n]$ with lexicographic  edge-labels according to Theorem \ref{ajaj}, Proposition \ref{gen:non} follows.
\end{proof}

  \begin{theorem} \label{gen:ajaj} Let  $T^l$ be a well-structured graph on the vertex set $[n]$.   The set of leaves of a $\mathcal{B}$-reduction tree  with root labeled $T^l$ is, up commutations,  the set of all    alternating well-structured spanning graphs $G$ of  $\overline{T^l}$  on the vertex set $[n]$ with lexicographic  edge-labels.
\end{theorem}

\proof The proof is analogous to that of Theorem \ref{ajaj} using Proposition \ref{gen:non} instead of   Proposition \ref{non}.

\qed

\medskip 

As corollaries of  Theorem \ref{gen:ajaj} we obtain  the  characterziation of reduced forms of the noncommutative monomial $m^\mathcal{B}[T^l]$, a triangulation of $\mathcal{P}(T^l)$ and a way to compute its volume, for a well-structured graph  $T^l$ on the vertex set $[n]$.

\medskip

\begin{theorem} {\bf (Noncommutative part.)} If  the polynomial $P^\mathcal{B}_n(x_{ij}, y_{ij}, z_i)$ is a reduced form of  $m^\mathcal{B}[T^l]$ for a well-structured graph  $T^l$ on the vertex set $[n]$,  then up to commutations

$$P^\mathcal{B}_n(x_{ij}, y_{ij}, z_i)=\sum_G m^\mathcal{B}[G],$$

\noindent where the sum runs over all   alternating well-structured spanning graphs $G$ of  $\overline{T^l}$  on the vertex set $[n]$ with lexicographic  edge-labels.    
\end{theorem}

\medskip

\begin{theorem}
 {\bf (Commutative part.)}  If  the polynomial $P^{\mathcal{B}^c}_n(x_{ij}, y_{ij}, z_i)$ is a reduced form of  $m^{\mathcal{B}^c}[T^l]$ for  a well-structured graph  $T^l$ on the vertex set $[n]$,  then

  $$P^{\mathcal{B}^c}_n(x_{ij}=y_{ij}=z_i=1)=f_{T^l},$$

\noindent where $f_{T^l}$ is the number of alternating well-structured spanning graphs $G$ of  $\overline{T^l}$. 
\end{theorem}

   \begin{theorem}   {\bf (Triangulation and volume.)} 
Let $T^l_1, \ldots, T^l_m$ be all  alternating well-structured spanning graphs   of  $\overline{T^l}$ for a  well-structured graph  $T^l$ on the vertex set $[n]$. Then   $\mathcal{P}(T^l_1), \ldots, \mathcal{P}(T^l_m)$  are $n$-dimensional simplices forming a triangulation of $\mathcal{P}(T^l)$. Furthermore, 
$$\vol_n(\mathcal{P}(T^l))=f_{T^l}\frac{2}{n!},$$ \noindent where $f_{T^l}$ is the number of alternating well-structured spanning graphs $G$ of  $\overline{T^l}$. 
\end{theorem}


\section{A more general noncommutative algebra $\g$}
\label{11}
 
 In this section we define the noncommutative algebra $\g$, which specializes to $\b$ when we set $\beta=0$. We prove analogs of the results presented so far for this more general algebra. We also provide a way for calculating Ehrhart polynomials for certain type $C_n$ root polytopes.

    Let the {\bf $\beta$-bracket algebra} $\g$ {\bf  of type $C_n$} be an associative algebra over $\mathbb{Q}[\beta]$, where $\beta$ is a variable (and a central element),   with a set of generators  $\{x_{ij}, y_{ij}, z_i \mid 1 \leq i\neq j\leq n\}$ subject to the following relations:
 
 (1) $x_{ij}+x_{ji}=0,$ $y_{ij}=y_{ji}$, for $i \neq j$,

(2) $z_i z_j=z_j z_i$

($3$) $x_{ij}x_{kl}= x_{kl}x_{ij}$, $y_{ij}x_{kl}= x_{kl}y_{ij}$, $y_{ij}y_{kl}= y_{kl}y_{ij}$, for  $i <j, k<l$ distinct.

($4$) $z_i x_{kl}=x_{kl} z_i$, $z_i y_{kl}=y_{kl} z_i$, for all $i\neq k, l$

(5) $x_{ij}x_{jk}=x_{ik}x_{ij}+x_{jk}x_{ik}+\beta x_{ik}$,  for  $1\leq i<j<k\leq n$,

($5'$) $x_{jk}x_{ij}=x_{ij}x_{ik}+x_{ik}x_{jk}+\beta x_{ik}$, for  $1\leq i<j<k\leq n$,

(6) $x_{ij}y_{jk}=y_{ik}x_{ij}+y_{jk}y_{ik}+\beta y_{ik}$,  for  $1\leq i<j<k\leq n$,

($6'$) $y_{jk}x_{ij}=x_{ij}y_{ik}+y_{ik}y_{jk}+\beta y_{ik}$, for  $1\leq i<j<k\leq n$,

(7) $x_{ik}y_{jk}=y_{jk}y_{ij}+y_{ij}x_{ik}+\beta y_{ij}$, for  $1\leq i<j<k\leq n$,

($7'$) $y_{jk}x_{ik}=y_{ij}y_{jk}+x_{ik}y_{ij}+\beta y_{ij}$, for  $1\leq i<j<k\leq n$,

(8) $y_{ik}x_{jk}=x_{jk}y_{ij}+y_{ij}y_{ik}+\beta y_{ij}$, for  $1\leq i<j<k\leq n$,

($8'$) $x_{jk}y_{ik}=y_{ij}x_{jk}+y_{ik}y_{ij}+\beta y_{ij}$, for  $1\leq i<j<k\leq n$,

(9) $x_{ij}z_j=z_i x_{ij}+ y_{ij} z_i + z_j y_{ij}+\beta z_{i}+\beta y_{ij}$, for $1 \leq i<j \leq n,$

($9'$) $z_jx_{ij}= x_{ij}z_i+  z_i y_{ij}+  y_{ij}z_j+\beta z_{i}+\beta y_{ij}$, for $1 \leq i<j \leq n.$

\medskip

Kirillov \cite{kir} made Conjecture 1 not just for $\b$, but for a more general   $\beta$-bracket algebra of type $C_n$, which is almost identical to $\g$; it   differs in a term in relations (9) and $(9')$.  We prove   the analogue of Conjecture 1  for $\g$.

Notice that the commutativization of $\g$ yields the relations of $\t$, except for relations (9) and $(9')$ of $\g$, which can be obtained by combining relations (6) and (7) of $\t$. Since the Reduction Lemma (Lemma \ref{reduction_lemma}) hold for $\t$, so does it for $\g$, keeping in mind that relations (9) and $(9')$ of $\g$ are obtained by combining relations (6) and (7)  of $\t$. As a result, we can think of  relations $(5)-(9')$ of $\g$ as operations   subdividing root polytopes into smaller polytopes and keeping track of their lower dimensional intersections. 

A {\bf $\r$-reduction tree} is analogous to an $\mathcal{S}$-reduction tree, just that the children of the nodes are obtained by the relations $(5)-(9')$ of $\g$, and now some nodes have five, and some nodes have three children. See Figure \ref{fig:beta} for an example. If $\mathcal{T}^{\r}$ is a $\r$-reduction tree with root labeled $G$ and  leaves labeled by graphs $G_1, \ldots, G_q$, then  \begin{equation} \label{red} \mathcal{P}^\circ(G)=\mathcal{P}^\circ(G_1)\cup \cdots \cup \mathcal{P}^\circ(G_q),\end{equation} by an analogue of the Reduction Lemma.

    \begin{figure}[htbp] 
\begin{center} 
\includegraphics[width=1.2\textwidth]{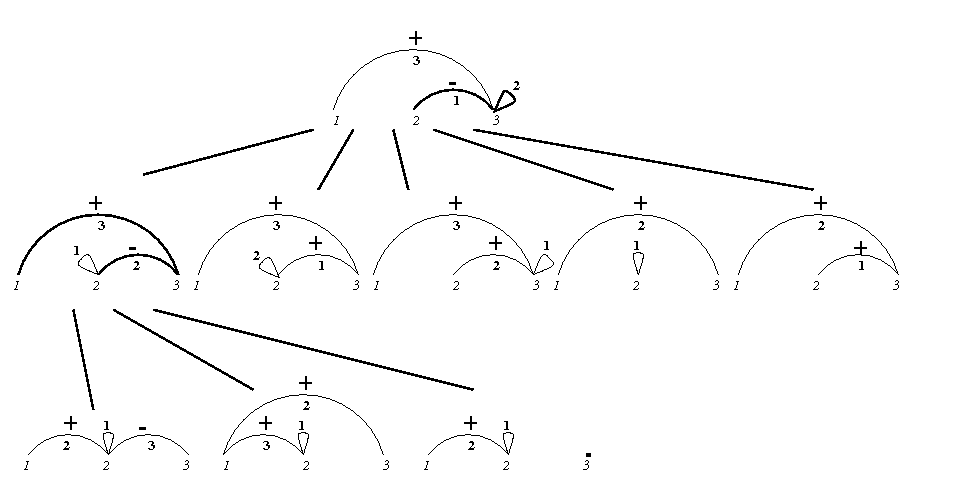} 
\caption{A $\r$-reduction tree with root corresponding to the monomial $x_{23}z_3y_{13}$.  Summing the monomials corresponding to  the graphs labeling the  leaves  of the reduction tree   multiplied by suitable powers of $\beta$, we obtain a reduced form $P^{\r}_n$ of  $x_{23}z_3y_{13}$,  $P^{\r}_n=z_2y_{12}x_{23}+z_2y_{13}y_{12}+\beta z_2y_{12}+y_{23}z_2y_{13}+z_3y_{23}y_{13}+ \beta z_2 y_{13}+\beta y_{23}y_{13}$. } 
\label{fig:beta}
\end{center} 
\end{figure} 

In order to prove an analogue of Proposition \ref{non} for the algebra $\g$, we need a definition more general  than well-structured. Thus we now define weakly-well-structured graphs.

A graph  $H$ on the vertex set $[n]$ and $p\leq n$ edges  is \textbf{weakly-well-structured} if it satisfies the following conditions:

$(i)$ $H$ is noncrossing. 

$(ii)$ For any two edges $(i, j, +), (k, l, +) \in H$, $i<j, k<l$, it must be that $i<l$ and $k<j$.

$(iii)$ For any two edges $(i, i, +), (k, l, +) \in H$, $k<l$, it must be that $k\leq i \leq l$ .


$(iv)$ There are no edges $(i, i, +), (k, j, -) \in H$ with $k<i<j$.

$(v)$ There are no edges $ (i, j, +), (k, l, -) \in H$ with $k\leq i<j \leq l$.



$(vi)$ Graph $H$  contains at most one loop, and $H$ contains no nonloop cycles.

$(vii)$ Graph $H$  contains a positive edge incident to vertex $1$.

             Note that well-structured graphs are also weakly-well-structured.

    \begin{proposition} \label{prop:P^l} By choosing the set of reductions suitably, the set of leaves of a $\r$-reduction tree  $\mathcal{T}^{\r}$ with root labeled by $P^l=([n], \{(n,n, +), (i, i+1, -) \mid i \in [n-1]\})$ can be the set of all  alternating weakly-well-structured subgraphs $G$ of  $\overline{P^l}$ with lexicographic  edge-labels.  
    \end{proposition}        
    
    \proof 
The proof of Proposition \ref{prop:P^l}  proceeds analogously as that of Proposition  \ref{non}, using  equation (\ref{red}), instead of the original statement of the Reduction Lemma, and using the full statement of \cite[Theorem 18]{kar} which says that the leaves of a reduction tree with root labeled by  $([n], \{(i, i+1, -) \mid i \in [n-1]\})$ are all noncrossing alternating forests with negative edges on the vertex set $[n]$ containing edge $(1, n, -)$ with lexicographic edge-labels.  
 \qed

 \begin{theorem} \label{P^l} The set of leaves of a $\r$-reduction tree  $\mathcal{T}^{\r}$ with root labeled $P^l$ is, up commutations,  the set of all  alternating weakly-well-structured subgraphs $G$ of  $\overline{P^l}$ with lexicographic  edge-labels.
\end{theorem} 

\proof Proposition \ref{prop:P^l} proves the existence of one such $\r$-reduction tree. An analogue of Lemma \ref{huh} states that  if the root of a $\r$-reduction tree is a weakly-well-structured well-labeled graph, then so are all its nodes. Together with equation (\ref{red}) these imply Theorem \ref{P^l}.

\qed

As  corollaries of  Theorem \ref{P^l} we obtain  the  characterziation of reduced forms of the noncommutative monomial $w_{C_n}$ in $\g$ as well as a canonical triangulation of $\mathcal{P}(P^l)$ and an expression for its Ehrhart polynomial.

\begin{theorem} \label{beta-main}
If  the polynomial $P^{\r}_n (x_{ij}, y_{ij}, z_{i})$ is a reduced form of  $w_{C_n}$ in $\g$,  then

$$P^{\r}_n(x_{ij}, y_{ij}, z_{i})=\sum_{G} \beta^{n-|E(G)|}m^\mathcal{B}[G],$$

\noindent where the sum runs over all   alternating weakly-well-structured graphs $G$   on the vertex set $[n]$ with lexicographic  edge-labels.\end{theorem}

\begin{theorem} \label{canonical} {\bf (Canonical triangulation.)}  Let $G_1,\ldots, G_k$ be all the alternating  well-structured graphs on the vertex set $[n]$. Then the root polytopes $\mathcal{P}(G_1), \ldots, $ $\mathcal{P}(G_k)$ are   $n$-dimensional simplices forming a triangulation of   $\mathcal{P}(P^l)$. Furthermore, the   intersections of the top dimensional simplices   $\mathcal{P}(G_1), \ldots, \mathcal{P}(G_k)$  are simplices $ \mathcal{P}(H)$, where  $H$ runs over all alternating  weakly-well-structured graphs on the vertex set $[n]$. 
 \end{theorem}
 
 Given a polytope $\mathcal{P}\subset \mathbb{R}^{n}$, the {\bf $t^{th}$ dilate} of $\mathcal{P}$ is 
$$\displaystyle t \mathcal{P}=\{(tx_1, \ldots, tx_{n}) |  (x_1, \ldots, x_{n}) \in \mathcal{P}\}.$$

The {\bf Ehrhart polynomial of an integer polytope} $\mathcal{P}\subset \mathbb{R}^{n}$ is   
$$\displaystyle L_{\mathcal{P}} (t) = \# (t\mathcal{P} \cap \mathbb{Z}^{n}).$$

 For background on the  theory of Ehrhart  polynomials see \cite{br}.

  \begin{theorem} \label{ehrhart} {\bf (Ehrhart polynomial.)}    
  \begin{align*} L_{\mathcal{P}(P^l)}(t) &=(-1)^{n} \left(  \sum_{d=1}^{n} f^l(d) (-1)^d\left({d+t  \choose d}+{d+t-1 \choose d}\right)+\sum_{d=1}^{n-1} f(d) (-1)^d {d+t \choose d} \right) ,\end{align*} where  $f^l(d)$ is the number of alternating  weakly-well-structured graphs on the vertex set $[n]$ with $d$ edges one of which is a  loop and $f(d)$ is the number of alternating  weakly-well-structured graphs on the vertex set $[n]$ with $d$ edges and  no loops.

\end{theorem}

\proof
By Theorem \ref{canonical}, $\mathcal{P}(P^l)^\circ=\bigsqcup_{F \in W}  \mathcal{P}(F)^\circ \bigsqcup \bigsqcup_{F^l \in W^l}  \mathcal{P}(F^l)^\circ,$ where  $W$ is the set of  all alternating  weakly-well-structured graphs on the vertex set $[n]$ with no loops and $W^l$  is the set of   all alternating  weakly-well-structured graphs on the vertex set $[n]$ with a loop.  Then 
$$ \displaystyle  L_{\mathcal{P}(P^l)^\circ}(t)=\sum _{F \in W} L_{\mathcal{P}(F)^\circ}(t)+\sum _{F^l \in W^l} L_{\mathcal{P}(F^l)^\circ}(t).$$ By \cite[Theorem 1.3]{s1} the Ehrhart series of $\mathcal{P}(F)$, $F\in W$, $\#E(F)=d$, and  $\mathcal{P}(F^l)$, $F^l\in W^l$, $\#E(F^l)=d$, respectively, are 
$J(\mathcal{P}(F) , x)=1+\sum_{t=1}^\infty  L_{\mathcal{P}(F) }(t)x^t=\frac{1}{(1-x)^{d+1}}$ and $J(\mathcal{P}(F^l) , x)= \frac{1+x}{(1-x)^{d+1}}$. Equivalently, $L_{\mathcal{P}(F)^\circ}(t)={t-1 \choose d}$, $L_{\mathcal{P}(F^l)^\circ}(t)={t-1 \choose d}+{t \choose d}.$ Thus, 
$$ \displaystyle  L_{\mathcal{P}(P^l)^\circ}(t)=\sum_{d=1}^{n} f^l(d) \left({t-1 \choose d}+{t \choose d}\right)+\sum_{d=1}^{n-1} f(d) {t-1 \choose d},$$ where $f^l(d)=\# \{F^l\in W^l \mid \#E(F^l)=d\}$, $f(d)=\# \{F\in W \mid \#E(F)=d\}$. Using the Ehrhart-Macdonald reciprocity \cite[Theorem 4.1]{br} \begin{align*} & L_{\mathcal{P}(P^l)}(t) =(-1)^{n} L_{\mathcal{P}(P^l)^\circ}(-t)= \\ &=(-1)^{n} \left( \sum_{d=1}^{n} f^l(d) (-1)^d \left( {d+t  \choose d}+{d+t-1 \choose d} \right)+\sum_{d=1}^{n-1} f(d) (-1)^d {d+t \choose d}\right). \end{align*}
 
\qed

Theorems \ref{P^l}, \ref{beta-main}, \ref{canonical} and \ref{ehrhart} can be generalized to any well-structured graph $G$ by   adding further  technical requirements on the weakly-well-structured graphs that can appear among the leaves of a $\r$-reduction tree with root labeled by $G$. Due to the technical nature of these results, we omit them here.

\section{The type $D_n$ bracket algebra}
\label{sec:D}

 In the rest of the paper   we study the reduced forms of elements in the type $D_n$ bracket algebra with combinatorial methods fused with noncommutative Gr\"obner basis theory. While the connection with subdivisions of type $C_n$ root polytopes is present in this case as well, for brevity we choose to suppress this aspect.    

  Let the {\bf $\beta$-bracket algebra} $\gd$ {\bf  of type $D_n$} be an associative algebra over $\mathbb{Q}[\beta]$, where $\beta$ is  a variable (and a central element),   with a set of generators  $\{x_{ij}, y_{ij} \mid 1 \leq i\neq j\leq n\}$ subject to the following relations:
 
 (1) $x_{ij}+x_{ji}=0,$ $y_{ij}=y_{ji}$, for $i \neq j$,

(2) $z_i z_j=z_j z_i$

($3$) $x_{ij}x_{kl}= x_{kl}x_{ij}$, $y_{ij}x_{kl}= x_{kl}y_{ij}$, $y_{ij}y_{kl}= y_{kl}y_{ij}$, for  $i <j, k<l$ distinct.

($4$) $z_i x_{kl}=x_{kl} z_i$, $z_i y_{kl}=y_{kl} z_i$, for all $i\neq k, l$

(5) $x_{ij}x_{jk}=x_{ik}x_{ij}+x_{jk}x_{ik}+\beta x_{ik}$,  for  $1\leq i<j<k\leq n$,

($5'$) $x_{jk}x_{ij}=x_{ij}x_{ik}+x_{ik}x_{jk}+\beta x_{ik}$, for  $1\leq i<j<k\leq n$,

(6) $x_{ij}y_{jk}=y_{ik}x_{ij}+y_{jk}y_{ik}+\beta y_{ik}$,  for  $1\leq i<j<k\leq n$,

($6'$) $y_{jk}x_{ij}=x_{ij}y_{ik}+y_{ik}y_{jk}+\beta y_{ik}$, for  $1\leq i<j<k\leq n$,

(7) $x_{ik}y_{jk}=y_{jk}y_{ij}+y_{ij}x_{ik}+\beta y_{ij}$, for  $1\leq i<j<k\leq n$,

($7'$) $y_{jk}x_{ik}=y_{ij}y_{jk}+x_{ik}y_{ij}+\beta y_{ij}$, for  $1\leq i<j<k\leq n$,

(8) $y_{ik}x_{jk}=x_{jk}y_{ij}+y_{ij}y_{ik}+\beta y_{ij}$, for  $1\leq i<j<k\leq n$,

($8'$) $x_{jk}y_{ik}=y_{ij}x_{jk}+y_{ik}y_{ij}+\beta y_{ij}$, for  $1\leq i<j<k\leq n$,



\medskip

Note that $\g$ is the quotient of $\gd$, since $\g$ has all the above relation and in addition relations (9), ($9'$); see Section \ref{11}.

Let $w_{D_n}=\prod_{i=1}^{n-1} x_{i, i+1}y_{n-1, n}$ be a Coxeter type element in $\gd$ and let $P^\mathcal{B}_n$ be the polynomial in variables $x_{ij}, y_{ij},   1 \leq i\neq j\leq n$  obtained from $w_{D_n}$ by successively applying the defining relations $(5)-(8')$ in any order until unable to do so, in the algebra $\q[\beta]\langle x_{ij}, y_{ij} \mid 1\leq i<j\leq n \rangle / I,$ where $I$ is the (two-sided) ideal generated by the relations $(1)-(4)$. 
We call $P^\mathcal{B}_n$  a {\bf reduced form} of $w_{D_n}$ and consider the process of  successively applying the defining relations $(5)-(8')$ as a reduction process in  $\q[\beta]\langle x_{ij}, y_{ij} \mid 1\leq i<j\leq n  \rangle / I,$ with the {\bf reduction rules}:

\medskip

(5) $x_{ij}x_{jk}\rightarrow x_{ik}x_{ij}+x_{jk}x_{ik}+\beta x_{ik}$,  for  $1\leq i<j<k\leq n$,

\medskip

($5'$) $x_{jk}x_{ij}\rightarrow x_{ij}x_{ik}+x_{ik}x_{jk}+\beta x_{ik}$, for  $1\leq i<j<k\leq n$,

\medskip

(6) $x_{ij}y_{jk}\rightarrow  y_{ik}x_{ij}+y_{jk}y_{ik}+\beta y_{ik}$,  for  $1\leq i<j<k\leq n$,

\medskip

($6'$) $y_{jk}x_{ij} \rightarrow x_{ij}y_{ik}+y_{ik}y_{jk}+\beta y_{ik}$, for  $1\leq i<j<k\leq n$,

\medskip

(7) $x_{ik}y_{jk} \rightarrow y_{jk}y_{ij}+y_{ij}x_{ik}+\beta y_{ij}$, for  $1\leq i<j<k\leq n$,

\medskip

($7'$) $y_{jk}x_{ik} \rightarrow y_{ij}y_{jk}+x_{ik}y_{ij}+\beta y_{ij}$, for  $1\leq i<j<k\leq n$,

\medskip

(8) $y_{ik}x_{jk}\rightarrow x_{jk}y_{ij}+y_{ij}y_{ik}+\beta y_{ij}$, for  $1\leq i<j<k\leq n$,

\medskip

($8'$) $x_{jk}y_{ik} \rightarrow y_{ij}x_{jk}+y_{ik}y_{ij}+\beta y_{ij}$, for  $1\leq i<j<k\leq n$.

\medskip

The reduced form of any other element of $\gd$ is defined analogously.  As in the type $C_n$ case, the relations of $\gd$ can be interpreted as subdividing type $C_n$  root polytopes and the reduced form of an element as a subdivision, though not a triangulation, of a type $C_n$ polytope. We pursue a different approach to studying reduced forms here. 

We can think  of the reduction process   in $\q\langle \beta, x_{ij}, y_{ij} \mid 1\leq i<j\leq n \rangle / I_{\beta},$ where the generators of the (two-sided) ideal $I_{\beta}$ are those of $I$ and in addition the commutators of $\beta$ with all the other variables $x_{ij}, y_{ij}$, $ 1\leq i<j\leq n$.

 \medskip
\noindent {\bf Conjecture 2. (Kirillov \cite{kir})}\textit{ Apart from applying the  relations (1)-(4), the reduced form  $P^\mathcal{B}_n$ of $w_{D_n}$ does not depend on the order in which the reductions are performed.}
 \medskip

Note that the above statement does not hold true for any monomial; some examples illustrating this were already explained in the comments after Conjecture 1 in Section \ref{sec:alg}.

\section{Graphs for type $D_n$}
\label{sec:13}

       It is straighforward to reformulate the  reduction rules (5)-($8'$) in terms of reductions on graphs. If $m \in \gd$, then we replace each monomial $m$ in the reductions by corresponding graphs $G^\mathcal{B}[m]$. The analogous procedure for type $C_n$ is explained in detail in Section \ref{reductionsB}.  
   
   We now define a central notion for those signed graphs whose corresponding monomials turn out to have a unique reduced form in $\gd$.  We reuse the expression ``good graph" from the type $C_n$ case, though the meaning in type $D_n$ is different. Previously we used good in the type $C_n$ sense; in the following we use good in the type $D_n$ sense.

A graph  $H$ on the vertex set $[n]$ and $k$ edges labeled   $1, \ldots, k$  is \textbf{good} if it satisfies the following conditions:

 $(i)$ The negative edges of $H$ form a noncrossing graph.
  
  $(ii)$ If edges $(i, j, -)_a$ and $(j, k, \e_2)_b$ are in $H$, $i<j<k$, $\e_2 \in \{-, +\}$,  then $a<b$.

 $(iii)$ If edges $(i, j, -)_a$ and $(i, k, \e_2)_b$ are in $H$, $i<j<k$, $\e_2 \in \{-, +\}$,  then $a>b$.
 
 $(iv)$ If edges $( j,k,  -)_a$ and $(i, k, \e_2)_b$ are in $H$, $i<j<k$, $\e_2 \in \{-, +\}$,  then $a<b$.

 $(v)$ If edges $(j,k,  +)_a$ and $(i, k, -)_b$ are in $H$, $i<j<k$,    then $a>b$.

$(vi)$  If edges $(i,k,  -)_a$ and $(j, l, +)_b$ are in $H$, $i<j<k<l$,    then $a>b$.

 \medskip
 
 \begin{lemma} \label{good}
 If $H$ is a good graph, then reduction rules $(5'), (6'), (7'), (8)$ cannot be performed on it. If we perform any of   the reduction rules $(5), (6), (7),(8')$ on $H$, then we obtain a graph $H^r$, which is also a good graph.
 \end{lemma}
 
 \proof  Note  that there is no way of commuting the labels of good graphs as to obtain an order on the edges which would allow rules $(5'), (6'), (7'), (8)$ to be performed. 
 
 That the following properties carry over from $H$ to $H^r$ follows from \cite[Lemma 12]{m1}, noting that  only reduction rule $(5)$ creates new negative edges:
 \begin{itemize}
 
\item The negative edges of $H$ form a noncrossing graph.
  
 \item If edges $(i, j, -)_a$ and $(j, k, -)_b$ are in $H$, $i<j<k$,   then $a<b$.

\item  If edges $(i, j, -)_a$ and $(i, k, -)_b$ are in $H$, $i<j<k$,   then $a>b$.
 
\item If edges $( j,k,  -)_a$ and $(i, k, -)_b$ are in $H$, $i<j<k$,   then $a<b$.

  \end{itemize}

 Inspection shows that the following properties  carry over from $H$ to $H^r$, keeping in mind that the above properties carry over for negative edges.
 
  \begin{itemize}
 
\item If edges $(i, j, -)_a$ and $(j, k, +)_b$ are in $H$, $i<j<k$,    then $a<b$.

 \item  If edges $( j,k,  -)_a$ and $(i, k, +)_b$ are in $H$, $i<j<k$,   then $a<b$.

\item If edges $(j,k,  +)_a$ and $(i, k, -)_b$ are in $H$, $i<j<k$,   then $a>b$.

\item  If edges $(i,k,  -)_a$ and $(j, l, +)_b$ are in $H$, $i<j<k<l$,    then $a>b$.
  \end{itemize}

Finally, given that all the above properties carry over from $H$ to $H^r$, it follows that the property 
 \begin{itemize}
 
\item  If edges $(i, j, -)_a$ and $(i, k, +)_b$ are in $H$, $i<j<k$,   then $a>b$.
 
\end{itemize}

also carries over. 

 \qed

Why are good graphs so good? Well, if the relations $(5), (6), (7),(8')$ were a noncommutative Gr\"obner basis for the ideal they generate in  $\q\langle \beta, x_{ij}, y_{ij} \mid 1\leq i<j\leq n \rangle / I_{\beta},$ with tips 
$x_{ij}x_{jk}$,  
 $x_{ij}y_{jk} $,   
 $x_{ik}y_{jk} $, 
 $x_{jk}y_{ik} $, respectively,
then it would follow immediately that the reduced form of monomials corresponding to good graphs are unique by results in noncommutative Gr\"obner bases theory. As it turns out the previous is not the case, however, we can still use Gr\"obner bases to prove the uniqueness of the reduced forms of the monomials corresponding to good graphs, which we call {\bf good monomials}, with a little bit more work. We show how to do this in the next section.

\section{Gr\"obner bases}
\label{sec:grobi}

In this section we briefly review some facts about noncommutative Gr\"obner bases and use them to show that the reduced forms of good monomials are unique.  

 We use the terminology and notation of \cite{g}, but state the results only for our special algebra. For the more general statements, see \cite{g}. Throughout this section  we consider the noncommutative case only. 
 
Let  $${\bf R=\q\langle \beta, x_{ij}, y_{ij} \mid 1\leq i<j\leq n \rangle / I_{\beta}}$$ with multiplicative basis $\base$, the set of noncommutative monomials up to equivalence under the commutativity relations described by $I_\beta$. 

The {\bf tip} of an element $f \in R$ is the largest basis element appearing in its expansion, denoted by Tip$(f)$. Let CTip$(f)$ denote the coefficient of Tip$(f)$ in this expansion. A set of elements $X$ is {\bf tip reduced} if for distinct elements $x, y \in X$, Tip$(x)$ does not divide Tip$(y)$.

\medskip

A well-order $>$ on $\base$ is {\bf admissible} if for $p, q, r, s \in \base$:

1. if $p<q$ then $pr<qr$ if both $pr\neq 0$ and $qr \neq 0$;

2. if $p<q$ then $sp<sq$ if both $sp\neq 0$ and $sq \neq 0$;

3. if $p=qr$, then $p>q$ and $p>r$.

\medskip

Let $f, g \in R$  and suppose that there are monomials $b, c \in \base$ such that 

\medskip

1. Tip$(f)c$=$b$Tip$(g)$.

\medskip

2. Tip$(f)$ does not divide $b$ and Tip$(g)$ does not divide $c.$

\medskip
Then the {\bf overlap relation of $f$ and $g$ by $b$ and $c$} is 

$$o(f, g, b, c)=\frac{fc}{\mbox{CTip}(f)}-\frac{bg}{\mbox{CTip}(g)}.$$


\medskip 
 \begin{proposition} \label{suf} (\cite[Theorem 2.3]{g}) A tip reduced  generating set of elements $\gr$ of the ideal $J$ of $R$ is a Gr\"obner basis, where the ordering on the monomials is admissible, if for every overlap relation $$o(g_1, g_2, p, q) \Rightarrow_{\gr} 0,$$

\noindent where $g_1, g_2 \in \gr$ and the above notation means that dividing $o(g_1, g_2, p, q)$ by $\gr$ yields a remainder of $0$. 
\end{proposition}
See \cite[Theorem 2.3]{g} for the more general formulation of Proposition \ref{suf} and  \cite[Section 2.3.2]{g} for the formulation of the Division Algorithm.

\begin{proposition} \label{grobi} Let $J$ be the ideal generated by the elements 
\begin{itemize}
\item  $x_{ij}x_{jk}-x_{ik}x_{ij}-x_{jk}x_{ik}-\beta x_{ik}$,  for  $1\leq i<j<k\leq n$,

\item $x_{ij}y_{jk}-y_{ik}x_{ij}-y_{jk}y_{ik}-\beta y_{ik}$,  for  $1\leq i<j<k\leq n$,

\item $x_{ik}y_{jk}-y_{jk}y_{ij}-y_{ij}x_{ik}-\beta y_{ij}$, for  $1\leq i<j<k\leq n$,

\item $x_{jk}y_{ik}-y_{ij}x_{jk}-y_{ik}y_{ij}-\beta y_{ij}$, for  $1\leq i<j<k\leq n$,

\end{itemize}

\noindent in $R / Y,$ where   $Y$ is the ideal in  $R$ generated by the elements 
\begin{itemize}
\item $x_{ik}x_{ij}y_{ik}+x_{jk}x_{ik}y_{ik}+\beta x_{ik}y_{ik}-x_{ij}y_{ij}x_{jk}-x_{ij}y_{ik}y_{ij}-\beta x_{ij}y_{ij}$, for  $1\leq i<j<k\leq n$. 

\end{itemize}

\noindent Then there is a monomial order in which the above generators of $J$ form a Gr\"obner basis $\gr$ of $J$ in  $R/ Y,$ and the tips of the generators are, respectively, 
\begin{itemize}
\item  $x_{ij}x_{jk}$,  
 
\item $x_{ij}y_{jk} $,   
 
\item $x_{ik}y_{jk} $, 
 
\item $x_{jk}y_{ik} $.

\end{itemize}

\end{proposition}

\proof Let $x_{ij}>y_{kl}$ for any $i<j$, $k<l$, and let $x_{ij}>x_{kl}$ and $y_{ij}>y_{kl}$ if $(i, j)$ is less than $(k, l)$ lexicographically. The degree of a monomial is determined by setting the degrees of $x_{ij}, y_{ij}$ to be $1$ and the degrees of $\beta$ and scalars to be $0$. A monomial with higher degree is bigger in the order $>$, and the lexicographically bigger monomial of the same degree is greater than the lexicographically smaller one. Since in $R$ two equal monomials can be written in two different ways due to commutations, we can pick a representative to work with, say the one which is the ``largest" lexicographically among all possible ways of writing the monomial, to resolve any ambiguities.  The order $>$ just defined is admissible, in it the tips of  
\begin{itemize}
\item  $x_{ij}x_{jk}-x_{ik}x_{ij}-x_{jk}x_{ik}-\beta x_{ik}$,  for  $1\leq i<j<k\leq n$,

\item $x_{ij}y_{jk}-y_{ik}x_{ij}-y_{jk}y_{ik}-\beta y_{ik}$,  for  $1\leq i<j<k\leq n$,

\item $x_{ik}y_{jk}-y_{jk}y_{ij}-y_{ij}x_{ik}-\beta y_{ij}$, for  $1\leq i<j<k\leq n$,

\item $x_{jk}y_{ik}-y_{ij}x_{jk}-y_{ik}y_{ij}-\beta y_{ij}$, for  $1\leq i<j<k\leq n$,

\end{itemize}

are 

\begin{itemize}
\item  $x_{ij}x_{jk}$,  
 
\item $x_{ij}y_{jk} $,   
 
\item $x_{ik}y_{jk} $, 
 
\item $x_{jk}y_{ik} $.

\end{itemize}

In particular the generators of $J$ are tip reduced. A calculation of the overlap relations shows that  $o(g_1, g_2, p, q) \Rightarrow_{\gr} 0$ in $R/ Y$,   where $g_1, g_2 \in \gr$. 
Proposition \ref{suf} then implies Proposition \ref{grobi}.
\qed

\begin{corollary} \label{/Y}
The reduced form of a good monomial $m$  is unique in $R / Y$. 
\end{corollary}

\proof Since the tips of elements of the Gr\"obner basis $\gr$ of  $J$ are exactly the monomials which we replace in the prescribed reduction rules $(5), (6), (7), (8')$, the reduced form of a good monomial $m$ is the remainder $r$ upon division by the elements of $\gr$ with the order $>$ described in the proof of Proposition \ref{grobi}. Since we proved that in $R/Y$ the basis $\gr$ is a Gr\"obner basis f $J$, it follows 
by \cite[Proposition 2.7]{g} that the remainder $r$ of the division of $m$ by $\gr$ is unique in $R/Y$.  That is, the reduced form of a good monomial $m$  is unique in $R / Y$. 
\qed

We would, however, like to prove uniqueness of the reduced form of a good monomial $m$ in  $R$. This is what the next series of statements   accomplish.

\begin{lemma} \label{Y}
There is a monomial order in which the elements 
\begin{itemize}
\item $x_{ik}x_{ij}y_{ik}+x_{jk}x_{ik}y_{ik}+\beta x_{ik}y_{ik}-x_{ij}y_{ij}x_{jk}-x_{ij}y_{ik}y_{ij}-\beta x_{ij}y_{ij}$, for  $1\leq i<j<k\leq n$, 

\end{itemize}

\noindent  are a Gr\"obner basis of $Y$ in $R$, and the tip of  $x_{ik}x_{ij}y_{ik}+x_{jk}x_{ik}y_{ik}+\beta x_{ik}y_{ik}-x_{ij}y_{ij}x_{jk}-x_{ij}y_{ik}y_{ij}-\beta x_{ij}y_{ij}$ is $x_{ij}y_{ik}y_{ij}$.
\end{lemma}

\proof  Let  $x_{ij}<y_{kl}$ for any $i<j$, $k<l$, and let $x_{ij}>x_{kl}$ and $y_{ij}>y_{kl}$ if $(i, j)$ is less than $(k, l)$ lexicographically. The degree of a monomial is determined by setting the degrees of $x_{ij}, y_{ij}$ to be $1$ and the degrees of $\beta$ and scalars to be $0$. A monomial with higher degree is bigger in the order $>$, and the lexicographically bigger monomial of the same degree, the variables being read from left to right,   is greater than the lexicographically smaller one. Since in $R$ two equal monomials can be written in two different ways due to commutations, we can pick a representative to work with, say the one which is the ``largest" lexicographically among all possible ways of writing the monomial, to resolve any ambiguities.  The order $>$ just defined is admissible, the tip of  $x_{ik}x_{ij}y_{ik}+x_{jk}x_{ik}y_{ik}+\beta x_{ik}y_{ik}-x_{ij}y_{ij}x_{jk}-x_{ij}y_{ik}y_{ij}-\beta x_{ij}y_{ij}$ is $x_{ij}y_{ik}y_{ij}$, and thus the generators of $Y$ are tip reduced. Since there are no overlap relations at all, by Proposition \ref{suf} Lemma \ref{Y} follows.

\qed

\begin{corollary} \label{div}
If $f\in Y$ then there is a term of   $f$  which can be written as  $m_1 \cdot x_{ij}y_{ik}y_{ij} \cdot m_2$ for some $1\leq i<j<k\leq n$, where $m_1, m_2$ are some monomials in $R$.
\end{corollary}

\proof Lemma \ref{Y} implies that  $$\langle x_{ij}y_{ik}y_{ij} \mid 1\leq i<j<k\leq n\rangle=\langle Tip(Y) \rangle.$$ From here the statement follows.

\qed

\begin{theorem} \label{unique}
The reduced form of a good monomial $m$  is unique in $R$. 
\end{theorem}

\proof By Corollary \ref{/Y} the reduced form of a good monomial $m$  is unique in $\q\langle \beta, x_{ij}, y_{ij} \mid 1\leq i<j\leq n \rangle / I_{\beta} / Y$.  Since by Corollary \ref{div} every $f\in Y$ contains a term divisible by $x_{ij}y_{ik}y_{ij}$ for some $1\leq i<j<k\leq n$, it follows that the reduced form of a good monomial $m$  is unique in $\q\langle \beta, x_{ij}, y_{ij} \mid 1\leq i<j\leq n \rangle / I_{\beta},$ since a good monomial cannot contain any term divisible by $x_{ij}y_{ik}y_{ij}$ because of property $(iii)$, with $\e_2=+$. 
\qed
\medskip

A special case of Theorem \ref{unique} is the statement of  Conjecture 2, since $w_{D_n}$ is a good monomial.

  \section*{Acknowledgement}
 I am grateful to my advisor Richard Stanley for suggesting  this problem and for many helpful suggestions. I would  like to thank Alex Postnikov   for sharing his insight into root polytopes. I would also like to thank    Anatol Kirillov for his intriguing  conjectures.

\end{document}